\documentclass[11pt]{amsart}

\usepackage{latexsym}
\usepackage{amssymb}
\usepackage{amsmath}
\usepackage{color}

\usepackage{mathtools}

\calclayout

\newtheorem{theorem}{Theorem}[section]
\newtheorem{lemma}[theorem]{Lemma}
\newtheorem{proposition}[theorem]{Proposition}
\newtheorem{corollary}[theorem]{Corollary}

\newtheorem{remark}[theorem]{Remark}

\newcommand\id{\mathop{\rm id}}
\newcommand\tr{\mathop{\rm tr}}

\newcommand{\B}{\mathcal{B}}

\newcommand\nph{\varphi}

\newcommand{\cl}[1]{\mathcal{#1}}
\newcommand{\bb}[1]{\mathbb{#1}}

\newcommand{\sca}[1]{\langle#1\rangle}

\begin{document}

\title{Absolutely dilatable bimodule maps}


\author[A. Chatzinikolaou]{Alexandros Chatzinikolaou}
\address{
Department of Mathematics, National and Kapodistrian University of Athens, Ath\-ens 157 84, Greece}
\email{achatzinik@math.uoa.gr}

\author[I. G. Todorov]{Ivan G. Todorov}
\address{
School of Mathematical Sciences, University of Delaware, 501 Ewing Hall,
Newark, DE 19716, USA}
\email{todorov@udel.edu}

\author[L. Turowska]{Lyudmila Turowska}
\address{Department of Mathematical Sciences, Chalmers University
of Technology and the University of Gothenburg, Gothenburg SE-412 96, Sweden}
\email{turowska@chalmers.se}

\date{17 August 2026}

\begin{abstract}
We characterise absolutely dilatable completely positive maps 
on the space of all bounded operators on a Hilbert space that are also bimodular
over a given von Neumann algebra as  
rotations by a suitable unitary on a larger Hilbert space followed by 
slicing along the trace of an additional ancilla. We 
define the local, quantum and approximately quantum 
types of absolutely dilatable maps, 
according to the type of the 
admissible ancilla. We show that the local 
absolutely dilatable maps
admit an exact factorisation through an abelian ancilla and 
show that they are limits in the point weak* topology 
of conjugations by unitaries in the commutant of the 
given von Neumann algebra. 
We show that the Connes Embedding Problem is equivalent to 
deciding if all absolutely dilatable maps are approximately quantum. 
\end{abstract}

\maketitle



\section{Introduction}\label{s_intro}

Factorisable maps between tracial von Neumann algebras were introduced in the context of non-commutative ergodic theory 
by C. Anantharaman-Delaroche in \cite{ade}. 
These maps are unital, trace preserving and completely positive, 
and thus, in the case where the von Neumann algebras 
are full matrix algebras, they are 
quantum channels in the sense of (finite dimensional)
quantum information theory. 
Factorisable unital quantum channels on the algebra 
$M_n$ of all $n$ by $n$ matrices were 
characterised by U. Haagerup and M. Musat in 
\cite{hm1} as the maps arising from a conjugation 
by a unitary operator in a larger von Neumann algebra, obtained 
by $M_n$ after tensoring with an auxiliary finite tracial von Neumann 
algebra, 
called the ancilla. 
These authors characterised further the factorisable 
quantum channels admitting abelian ancillas, 
as well as the factorisable Schur channels
(that is, quantum channels that are also Schur 
multipliers), giving an example 
that resolved the asymptotic Birkhoff conjecture 
in quantum information theory \cite{svw} in the negative. 
According to \cite[Proposition 2.8]{hm1}, factorisable Schur multipliers 
on $M_n$ correspond precisely to $n\times n$ matrices of the form $[\tau(d_i^*d_j)]_{i,j}$, where $\{d_i\}_{i=1}^n$ is a set of unitary operators 
in a finite tracial von Neumann algebra $(\cl N,\tau)$. 
The paper \cite{hm1} shows as well that factorisability is 
equivalent to the absolute dilatability of the map, a property involving 
a simultaneous factorisation of all non-negative powers thereof. 

Schur multipliers have been studied beyond and before the context of 
quantum information theory; arising in the foundational 
work of I. Schur \cite{schur}, and 
following a characterisation obtained by A. Grothendieck in his
{\it R$\acute{e}$sum$\acute{e}$} \cite{Gro}, 
they have been used in perturbation theory, 
operator integral theory and non-commutative analysis, among others
(see \cite{Pa, jtt, tt-survey} and the references therein). 
In general, given a (standard) measure space $(\Omega,\sigma)$, 
Schur multipliers 
arise from the pointwise multiplication of integral operator kernels 
by a fixed measurable 
\emph{symbol} $\nph : \Omega\times \Omega\to \bb{C}$; 
they can equivalently be thought of as (completely) bounded weak* continuous maps 
on the space $\cl B(H)$ of all bounded linear operators on the  
Hilbert space $H = L^2(\Omega,\sigma)$ that are bimodular over the 
multiplication algebra of $L^{\infty}(\Omega,\sigma)$. 
A characterisation of the absolutely dilatable Schur multipliers 
over a measure space $(\Omega,\sigma)$ was obtained by C. Duquet and C. Le Merdy, 
extending the Haagerup-Musat characterisation to the elegant 
representation of the symbol $\nph(s,t) = \tau(d(s)^*d(t))$, 
where $d : \Omega\to \cl N$ is a unitary-valued measurable map. 
Further properties and connections with Fourier multipliers were 
given in \cite{duquet}. 

The main result of the present paper is a non-commutative version 
of Duquet-Le Merdy's theorem; this is obtained in Section \ref{s_dilmod}. 
We replace the modularity over a maximal abelian 
von Neumann algebra by a modularity over a 
von Neumann subalgebra $\cl D'\subseteq \cl B(H)$ 
(we denote by $\cl D'$ the commutant of a von Neumann algebra $\cl D$). 
Observing that measurable unitary-valued functions $d : X\to \cl N$
correspond to unitary operators in the von Neumann algebra 
$L^{\infty}(\Omega,\sigma)\bar\otimes \cl N$, we show that a 
unital completely positive $\cl D'$-modular map $\Phi : \cl B(H)\to \cl B(H)$
is absolutely dilatable 
if and only if 
it admits a representation of the form 
$$\Phi(z) = (\id\hspace{-0.05cm}\mbox{}_{\cl B(H)}\otimes \tau)(D^*(z\otimes 1_{\cl N})D),$$
where $D\in \cl D\bar\otimes\cl B(K)$ is a 
unitary satisfying natural requirements 
and $\cl N\subseteq \cl B(K)$ is a 
von Neumann algebra equipped with a finite trace $\tau$. 
In addition, we present an equivalent description of $\Phi$ in terms of its 
\emph{operator symbol} \cite{jtt}.
We borrow techniques from \cite{dlem}, 
and show that a large part of the proof of Theorem 1.1 therein can be 
lifted to the non-commutative and modular setting by using 
discrete coordinates as opposed to the measurable coordinates 
employed in \cite{dlem}. 

The origin of a natural hierarchy for factorisable maps on $M_n$, 
depending on the ancilla type, 
can be found already in \cite{hm2}. 
In Section \ref{s_hier}, we exhibit such a hierarchy 
for absolutely dilatable 
$\cl D'$-bimodule maps, 
which is also new in the case where $\cl D = L^{\infty}(\Omega,\sigma)$. 
This is inspired by the hierarchy of no-signalling correlations,
consisting of local, quantum, approximately quantum and quantum commuting 
correlations (see e.g. \cite{lmprsstw}). We define the subclasses of \emph{locally}, \emph{quantum} and 
\emph{approximately quantum factorisable} 
maps, as classes of absolutely 
dilatable maps admitting 
ancilla of a specific type (abelian, finite dimensional, and 
Connes embeddable, respectively). 
Motivated by the formulation of the Connes Embedding Problem 
in terms of equality of the quantum commuting and the 
approximately quantum correlation classes 
\cite{oz, jnpp, fritz}, we show that the 
negative answer to Connes Embedding Problem, established in 
\cite{jnvwy}, is equivalent to the 
properness of the inclusion of the class of 
approximately quantum factorisable maps in that of 
all factorisable maps (see Theorem \ref{th_clova} for the precise statement). 
We note that a type hierarchy 
for factorisable channels on $M_n$
that are associated with quantum permutations
(as opposed to general block unitary matrices), was 
defined and examined in \cite{pr}.

\medskip





\noindent 
{\bf Acknowledgements. } 
We are grateful to Jason Crann and Christian Le Merdy 
for useful discussions on 
the topic of the paper.
The authors thank the anonymous referee for many
suggestions that led to an improvement of the exposition.
The second author was supported by NSF grants 
CCF-2115071 and DMS-2154459. The third author was supported by the Swedish Research Council project grant 2023-04555 and GS Magnusons Fond MF2022-0006. The research described in this paper was carried out within the framework of the National Recovery and Resilience Plan Greece 2.0, funded by the European Union - NextGenerationEU (Implementation Body: HFRI. Project name: 
Noncommutative analysis: operator systems and non-locality. 
HFRI Project Number: 015825).


\section{A characterisation theorem}\label{s_dilmod}

In this section, we prove the main result of the paper, Theorem \ref{tar1}.
We begin by setting notation, and recalling some basic definitions and facts about 
completely positive maps and dilatability.

Let $H$ be a Hilbert space. We denote by $\cl B(H)$ the space of 
all bounded linear operators on $H$, and by $\cl S_p(H)$ the Schatten $p$-class 
on $H$. We will only need the latter ideals in the cases $p = 1$ and $p = 2$.
We denote by ${\rm tr}$ the trace on $\cl S_1(H)$.
The commutant of a von Neumann algebra $\cl N\subseteq \cl B(H)$ will be denoted by
$\cl N'$, its predual by 
$\cl N_*$ and, for $z\in \cl N$ and $\omega\in \cl N_*$ we often write 
$\langle z,\omega\rangle_{\cl N,\cl N_*} = \omega(z)$; when the 
von Neumann algebra $\cl N$ is clear from the context, we use simply 
$\langle \cdot,\cdot\rangle$. 
The same notation will be employed for Hilbert space inner products 
(it will be clear from the context which of these two uses is intended). 
We recall the canonical identification $\cl S_1(H)^* = \cl B(H)$, 
implemented by trace duality 
$\langle x,y\rangle = {\rm tr}(xy)$, $x\in \cl B(H)$, $y\in \cl S_1(H)$. 
If $\tau_{\cl N}$ is a semi-finite normal faithful trace on $\cl N$, 
we call the pair
$(\cl N,\tau_{\cl N})$ a \emph{tracial von Neumann algebra}.
If $\tau_{\cl N}$ is moreover finite and normalised so that $\tau_{\cl N}(1) = 1$, 
we call $(\cl N,\tau_{\cl N})$ \emph{finite}.

The cone of all positive elements in a von Neumann algebra 
$\cl N$ will be denoted by $\cl N^+$, and 
$M_n(\cl N)$ will stand for the von Neumann algebra of all 
$n$ by $n$ matrices with entries in $\cl N$. 
A linear map $\Phi : \cl B(H)\to \cl B(H)$ is called \emph{positive}, if 
$\Phi(\cl B(H)^+)\subseteq \cl B(H)^+$, and \emph{completely positive}
if the map $\Phi^{(n)} : M_n(\cl B(H)) \to M_n(\cl B(H))$, given by 
$\Phi^{(n)}((x_{i,j})_{i,j}) = (\Phi(x_{i,j}))_{i,j}$, 
is positive for every $n\in \bb{N}$.
If $\cl D \subseteq \cl B(H)$ is a von Neumann algebra, 
a linear map $\Phi : \cl B(H)\to \cl B(H)$ is called \emph{$\cl D$-modular}
(or a \emph{$\cl D$-bimodule map}) if 
$$\Phi(axb) = a\Phi(x)b, \ \ \ x\in \cl B(H), \ a,b\in \cl D.$$

We have that the following are equivalent, for a linear map $\Phi$ 
(see \cite{bs}):
\begin{itemize}
\item[(i)] $\Phi$ is completely bounded, weak* continuous and $\cl D'$-modular; 

\item[(ii)] there exist sets $(a_i)_{i\in \bb{I}}\subseteq \cl D$ and 
$(b_i)_{i\in \bb{I}}\subseteq \cl D$, such that 
the series  
\begin{equation}\label{eq_twoser}
\sum_{i\in \bb{I}} a_i^*a_i \ \text{ and } \  
\sum_{i\in \bb{I}} b_i^*b_i
\end{equation}
are weak* convergent and 
\begin{equation}\label{eq_Phi}
\Phi(x) = \sum_{i\in \bb{I}} a_i^* x b_i, \ \ \ x\in \cl B(H),
\end{equation}
where the latter series is weak* convergent. 
\end{itemize}
Assuming that the map $\Phi$ has the representation (\ref{eq_Phi}), 
the series 
$\sum_{i\in \bb{I}} a_i^*\otimes b_i $
is weak* convergent as an element of 
the weak* Haagerup tensor product 
$\cl D\otimes_{\rm w^*h}\cl D$ \cite{bs}. 
Furthermore, the elements of $\cl D\otimes_{\rm w^*h}\cl D$
correspond canonically to elements of $\cl B(H \otimes H)$
(here, and in the sequel, we denote by $H\otimes K$ the Hilbertian tensor 
products of the Hilbert spaces $H$ and $K$).
This fact is contained in \cite[Corollary 3.8]{bs}, but we provide a sketch 
of a direct argument that will be used later. 
Assume that the series (\ref{eq_twoser}) are weak* convergent. 
For a finite set $F\subseteq \bb{I}$, let 
$u_F = \sum_{i\in F} a_i^*\otimes b_i$, 
viewed as a bounded operator on $H \otimes H$. 
Given $\xi,\xi',\eta,\eta' \in H$, we have that
\begin{eqnarray*}
\left |\langle u_F(\xi\otimes\eta),\xi'\otimes\eta'\rangle\right |
& \leq & 
\sum_{i\in F} \left |\langle a_i^*\xi,\xi'\rangle\right | 
\left |\langle b_i \eta,\eta'\rangle\right |
= 
\sum_{i\in F} \left |\langle \xi,a_i\xi'\rangle\right | 
\left |\langle b_i \eta,\eta'\rangle\right |\\
& \leq & 
\|\xi\|\|\eta'\| \sum_{i\in F} \|a_i\xi'\| \|b_i\eta\|\\
& \leq & 
\|\xi\|\|\eta'\| \left(\sum_{i\in F} \|a_i\xi'\|^2\right)^{\frac{1}{2}}  
\left(\sum_{i\in F}  \|b_i\eta\|^2\right)^{\frac{1}{2}}\\
& = & 
\|\xi\|\|\eta'\| \left(\sum_{i\in F} \langle a_i^*a_i\xi',\xi'\rangle\right)^{\frac{1}{2}}  
\left(\sum_{i\in F} \langle b_i^*b_i\eta,\eta\rangle\right)^{\frac{1}{2}}\\
& = & 
\|\xi\|\|\eta'\| \left\langle \left(\sum_{i\in F} a_i^*a_i\right)\xi',\xi'\right\rangle^{\frac{1}{2}}  
\left \langle \left(\sum_{i\in F} b_i^*b_i\right)\eta,\eta\right\rangle^{\frac{1}{2}}.
\end{eqnarray*}
It follows that the net 
$(\langle u_F\zeta_1,\zeta_2\rangle)_{F}$
is Cauchy for all vectors $\zeta_1,\zeta_2$ in the 
algebraic tensor product $H \odot H$.
The uniform boundedness of the net $(\|u_F\|)_{F}$ now implies that 
the net $(\langle u_F\zeta_1,\zeta_2\rangle)_{F}$
is Cauchy for all $\zeta_1,\zeta_2\in H \otimes H$, that is, 
the net $(u_F)_{F}$ is Cauchy in the weak operator topology. 
We let $u_{\Phi}$ be its weak limit, 
as an element of $ \cl B(H\otimes H)$ and call it the \emph{operator symbol} of $\Phi$. 
In the case $\Phi$ is completely positive, the representation 
(\ref{eq_Phi}) is achieved with $a_k = b_k$ for all $k$. 
We refer the reader to \cite{Pa} for further 
background on completely positive maps. 

Suppose that $J : \cl B(H)\to \cl M$ is 
a trace-preserving normal *-homomorphism. 
By trace-preservation and the fact that the non-commutative 
$L^1$-space, associated with $(\cl M,\tau_{\cl M})$, 
can be canonically identified with the predual $\cl M_*$, we 
have that the restriction of $J$ to $\cl S_1(H)$ takes values 
in $\cl M_*$, thus obtaining a canonical map
$J_1 : \cl S_1(H)\to \cl M_*$. 
A weak* continuous map 
$\Phi : \cl B(H)\to \cl B(H)$
is called \emph{absolutely dilatable} \cite{dlem} 
if there exists a von Neumann algebra $\cl M$, equipped with 
a semi-finite faithful trace $\tau_{\cl M}$, a $\tau_{\cl M}$-preserving normal *-automorphism 
$U : \cl M\to \cl M$, and 
a unital trace-preserving normal *-homomorphism $J : \cl B(H)\to \cl M$, 
such that
$$\Phi^n(x) = J_1^*\circ U^n\circ J(x), \ \ \ x\in \cl B(H), \ n\in \bb{Z}_+$$
(here $J_1^* : \cl M\to \cl B(H)$ is the adjoint of $J_1$; 
note it coincides with the unique weak* continuous 
conditional expectation from $\cl M$ onto $J(\cl B(H))$). 

We note that every absolutely dilatable map is automatically 
unital, completely positive and trace-preserving; 
in addition, the *-homomorphism $J$ is automatically injective.
If $ H$ is separable and $ \Phi$ is absolutely dilatable with the von Neumann algebra $ \cl M$ having a 
separable predual, we say that $\Phi $ admits a separable absolute dilation. Note that a von Neumann algebra  $\cl M$ has a separable predual if and only if there exists a separable Hilbert space  $H$ such that $\cl M\subseteq \cl B(H)$. 

In the following, we will use the standard leg notation for operators acting 
on the tensor product of two Hilbert spaces:
if $H$ and $K$ are Hilbert spaces and $D\in \cl B(H\otimes K)$, 
by $D_{2,3}$ we denote the operator $I_H\otimes D\in \cl B(H\otimes H\otimes K)$, 
and by $D_{1,3}$ -- the operator $(I_H\otimes\frak{f}^{-1})\circ (D\otimes I_H)\circ(I_H\otimes\frak{f})$, 
where $\frak{f} : H\otimes K\to K\otimes H$ is the unitary operator, given by 
$\frak{f}(\xi_1\otimes \xi_2) = \xi_2\otimes\xi_1$. 
If $\cl D\subseteq \cl B(H)$ and $\cl N\subseteq \cl B(K)$ are von Neumann 
algebras, we write as usual $\cl D\bar\otimes\cl N$ for their 
spatial weak* tensor product. If $\omega\in \cl D_*$ and $T\in \cl D\bar\otimes\cl N$, 
we let $L_{\omega} : \cl D\bar\otimes\cl N\to \cl N$ be the (linear) slice map,
given by $L_{\omega}(A\otimes X) = \omega(A)X$, $A\in \cl D$, $X\in \cl N$. 
The same notation will be used for slice maps 
along functionals on $\cl N$. 

We include two lemmas that will be needed in the proof of Theorem \ref{tar1} below. 

\begin{lemma}\label{l_symbolc}
Let $H$ be a Hilbert space with orthonormal basis $(e_i)_{i\in \bb{I}}$, $\epsilon_{i,j}=e_ie_j^*$, $i,j\in \bb{I}$ the corresponding matrix units, $ \cl D \subseteq \cl B(H)$ be a von Neumann algebra and 
$\Phi : \cl B(H)\to \cl B(H)$ be a weak* continuous $\cl D'$-modular completely bounded map. 
Then 
$$\langle\Phi(\epsilon_{i,j}),\epsilon_{k,l}\rangle = 
\langle u_{\Phi} (e_i\otimes e_k), e_l\otimes e_j\rangle, \ \ \ 
i,j,k,l \in \bb{I}.$$
\end{lemma}

\begin{proof}
Assuming that (\ref{eq_Phi}) holds, we have 
\begin{eqnarray*}
\langle \Phi(\epsilon_{i,j}), \epsilon_{k,l}\rangle 
& = & 
\left\langle \sum_{r\in \bb{I}} a_r^* (e_ie_j^*) b_r, e_ke_l^* \right\rangle
= 
\sum_{r\in \bb{I}} \left\langle (a_r^*e_i)(b_r^*e_j)^*, e_ke_l^* \right\rangle\\
& = & 
\sum_{r\in \bb{I}} \tr((a_r^*e_i)(b_r^*e_j)^* (e_ke_l^*))
= 
\sum_{r\in \bb{I}} \langle a_r^*e_i,e_l\rangle \langle e_k, b_r^*e_j\rangle\\
& = & 
\sum_{r\in \bb{I}} \langle a_r^*e_i,e_l\rangle \langle b_r e_k, e_j\rangle
= 
\left\langle u_{\Phi} (e_i\otimes e_k), e_l\otimes e_j\right\rangle.
\end{eqnarray*}
\end{proof}

Suppose that $D\in \cl B(H\otimes K)$ and let 
$\{\epsilon_{i,j}\}_{i,j\in \bb{I}}$ be 
a matrix unit system in $\cl B(H)$ 
arising from an orthonormal basis. 
If $d_{i,j} = L_{\epsilon_{j,i}}(D)$, $i,j\in \bb{I}$, we have that 
$D=\sum_{i,j\in \bb{I}} \epsilon_{i,j}\otimes d_{i,j}$.

\begin{lemma}\label{l_returns}
Let $H$ and $K$ be Hilbert spaces, $\cl N\subseteq \cl B(K)$ be a von Neumann algebra 
and $D\in \cl B(H\otimes K)$. 
Fix a matrix unit system $\{\epsilon_{i,j}\}_{i,j\in \bb{I}}$ in $\cl B(H)$ arising from an orthonormal basis, and write 
$d_{i,j} = L_{\epsilon_{j,i}}(D)$, $i,j\in \bb{I}$. 
The following are equivalent:
\begin{itemize}
\item[(i)] $d_{i,j}^*d_{k,l}\in \cl N$ for all $i,j,k,l\in\bb{I}$;

\item[(ii)] $L_{\omega_1}(D^*)L_{\omega_2}(D)\in \cl N$ for all 
$\omega_1,\omega_2\in \cl B(H)_*$; 

\item[(iii)] $D_{1,3}^*D_{2,3}\in \cl B(H)\bar\otimes\cl B(H)\bar\otimes\cl N$;

\item[(iv)]
$D^*(x\otimes I_K)D\in \cl B(H)\bar\otimes\cl N$ for every $ x\in \cl B(H)$. 
\end{itemize}
\end{lemma}

\begin{proof}
(i)$\Rightarrow$(ii) 
Fix $k,l\in \bb{I}$. The norm-weak*
continuity and linearity of the map $\omega\to L_{\omega}(D^*)$
imply that $L_{\omega_1}(D^*)d_{k,l}\in \cl N$. 
Now the norm-weak* continuity and linearity  of the map $\omega\to L_{\omega}(D)$ show that 
$L_{\omega_1}(D^*)L_{\omega_2}(D)\in \cl N$. 

(ii)$\Rightarrow$(i) is trivial.

(ii)$\Rightarrow$(iii)
By the slice map property of von Neumann algebras \cite{kraus}, 
it suffices to show that $L_{\sigma}(D_{1,3}^*D_{2,3})\in \cl N$,
for all $\sigma\in \cl B(H\otimes H)_*$. By the norm-weak* continuity and the linearity of the map 
$\sigma\to L_{\sigma}(D_{1,3}^*D_{2,3})$, it suffices to assume that 
$\sigma = \omega_1\otimes \omega_2$, where $\omega_1,\omega_2\in \cl S_1(H)$. 
The conclusion now follows from the identity
\begin{equation}\label{eq_2313}
L_{\omega_1\otimes \omega_2}(S_{1,3}T_{2,3}) 
= L_{\omega_1}(S)L_{\omega_2}(T), \ \ \ S,T\in \cl B(H\otimes K).
\end{equation}
To see (\ref{eq_2313}), note first that the identity is trivial in the case where 
$S = A\otimes X$ and $T = B\otimes Y$, 
for some $A,B\in \cl B(H)$ and $X,Y\in \cl B(K)$. 
By the weak* continuity of the slice map and of the one-sided 
operator multiplication, (\ref{eq_2313}) holds true for an elementary 
tensor $S = A\otimes X$ and an arbitrary $T\in \cl B(H\otimes K)$.
Using once again the weak* continuity of the slice map and of the one-sided 
operator multiplication, we arrive at (\ref{eq_2313}) in the stated generality. 

(iii)$\Rightarrow$(ii) is a direct consequence of (\ref{eq_2313}). 

(i)$\Rightarrow$(iv)
It is enough to see that $L_\omega(D^*(x\otimes I_K)D)\in\cl N$ for $\omega=\epsilon_{k,l}$ and $x=\epsilon_{i,j}$. 
We have
\begin{eqnarray}
&&L_{\epsilon_{k,l}}(D^*(\epsilon_{i,j}\otimes I_K)D)\nonumber\\&&=L_{\epsilon_{k,l}}((\epsilon_{l,l}\otimes I_K)D^*(\epsilon_{i,i}\otimes I_K)(\epsilon_{i,j}\otimes I_K)(\epsilon_{j,j}\otimes I_K)D(\epsilon_{k,k}\otimes I_K))\label{eq}\\&&=L_{\epsilon_{k,l}}(\epsilon_{l,k}
\otimes d_{i,l}^*d_{j,k})=d_{i,l}^*d_{j,k},\nonumber
\end{eqnarray} 
which belongs to $\cl N$ by the assumption. 

(iv)$\Rightarrow$(i) follows from reversing the steps in the previous 
paragraph. 
\end{proof}

Let $D\in \cl B(H\otimes K)$ be a unitary operator. 
If $\cl N\subseteq \cl B(K)$ is a von Neumann algebra, we say that 
$D$ {\it returns to} $\cl N$, if the equivalent conditions in Lemma \ref{l_returns} 
are satisfied. 
By Lemma  \ref{l_returns}, if this happens then
the map 
$\Phi_D : \cl B(H)\to \cl B(H)$, given by 
\begin{equation}\label{eq_PhiDd}
\Phi_D(x) = (\id\otimes\tau_{\cl N})(D^*(x\otimes I_K)D), \ \ \ 
x\in \cl B(H),
\end{equation}
is well-defined. 

The next result is a simultaneous
generalisation of \cite[Theorem 2.2]{hm1} and \cite[Theorem 1.1 and 7.1]{dlem}. 
We will say that a unitary operator $ D\subseteq \cl B(H \otimes K)$ 
which returns to $\cl N$
\emph{satisfies trace preservation} if
there exists a unit vector $e\in H$ such that 
$ (\tr\otimes \tau_{\cl N})\big(D^*(ee^*\otimes 1_{\cl N})D\big)=1$. 

\begin{theorem}\label{tar1}
Let $H$ be a Hilbert space and $\cl D\subseteq \cl B(H)$ be a von Neumann algebra. 
The following are equivalent, for a weak* continuous, unital, completely positive map 
$\Phi : \cl B(H)\to \cl B(H)$:

\begin{itemize}
\item[(i)] $\Phi$ is $\cl D'$-modular and absolutely dilatable; 

\item[(ii)] there exist 
a Hilbert space $K$, 
a finite tracial von Neumann algebra $(\cl N,\tau_{\cl N})$ 
acting on $K$, and  a
unitary operator $D\in \cl D\bar\otimes\cl B(K)$ 
that returns to $\cl N$ and satisfies trace preservation,
such that $\Phi = \Phi_D$;
\item[(iii)] 
there exist a
Hilbert space $K$, 
a finite tracial von Neumann algebra $(\cl N,\tau_{\cl N})$ 
acting on $K$, and a unitary operator $D\in \cl D\bar\otimes\cl B(K)$ 
that returns to $\cl N$ and satisfies trace preservation,
such that
$u_\Phi = (\id\otimes\id\otimes\tau_{\cl N})(D_{1,3}^*D_{2,3})$. 
\end{itemize}
\end{theorem}

\begin{proof}
(i) $\Rightarrow $ (ii) 
Let $ \Phi : \cl B(H)\to \cl B(H)$ be $\cl D'$-modular and absolutely dilatable. 
Thus, there exists a tracial von Neumann algebra $(\cl M,\tau_{\cl M})$, 
a trace preserving unital weak* continuous *-homomorphism $J : \cl B(H) \rightarrow \cl M$ and a trace preserving unital *-automorphism $U : \cl M \rightarrow \cl M$, such that 
 $ \Phi^{n}= J_{1}^{*} U^{n} J$ for every $ n \in \bb N$. 
 Let  $ \{\epsilon_{i,j}\}_{i,j \in \bb{I}} \subseteq  \B(H)$ be a matrix unit system arising from an orthonormal basis,
 and note that $\{ J(\epsilon_{i,j})\}_{i,j\in \bb{I}}$ is a matrix unit system in $ \cl M$. 
Fix $ i_0 \in \bb{I} $ and set $q = J(\epsilon_{i_0,i_0})$; then 
 \[ m_{i,j} := J(\epsilon_{i_0,i})mJ(\epsilon_{j,i_0})  \]
is in $ q \cl M q$ for all $ i,j$, $ m \in \cl M$, and the map
$$    \rho : \cl M  \rightarrow \cl B(H) \bar \otimes (q \cl M q), \ 
     m \mapsto \sum_{i,j\in \bb{I}} \epsilon_{i,j} \otimes m_{i,j},$$
is a *-isomorphism \cite[Proposition IV 1.8]{tak1} 
(we note that the series converges in the weak* topology). 
The restriction $ \tau_{1}$ of $ \tau_{\cl M}$ to $ q \cl M q$ is a semi-finite trace. 
We equip $B(H) \bar \otimes (q \cl M q)$ with the semi-finite trace $ \tr \otimes \tau_{1}$, and observe that $ \rho $ is trace-preserving. 
Indeed, if $ m \in \cl M^{+}$ then 
\begin{eqnarray*}
(\tr \otimes \tau_{1})(\rho(m)) 
& = & \sum_{i\in \bb{I}} \tau_{\cl M}(m_{i,i})
=
\sum_{i\in \bb{I}} \tau_{\cl M}(J(\epsilon_{i_0,i})mJ(\epsilon_{i,i_0}))\\
& = & 
\sum_{i\in \bb{I}} \tau_{\cl M}(J(\epsilon_{i,i_0})J(\epsilon_{i_0,i})m)
= \sum_{i\in \bb{I}} \tau_{\cl M}(J(\epsilon_{i,i})m)\\
& = & 
\tau_{\cl M}\left(\sum_{i\in \bb{I}}J(\epsilon_{i,i})m\right)
= \tau_{\cl M}(m).
\end{eqnarray*}
Note also that $ \rho(J(\epsilon_{i,j})) = \epsilon_{i,j} \otimes q$ for all $i,j \in \bb{I}$. It follows from the weak* continuity of $ J$ and $\rho$ that $ \rho (J(x)) = x\otimes q$ for all $ x\in \cl B(H)$ and, since $ J$  is trace-preserving,
\begin{align*}
\tau_{1}(q)   = \tau_{\cl M} ( J(\epsilon_{i_0,i_0})) = \tr( \epsilon_{i_0,i_0} ) =1.
\end{align*}

By setting $ \cl N_{1} = q \cl M q$ and changing $ J$ into $ \rho \circ J$ and $ U$ into $ \rho \circ U \circ \rho^{-1}$ we assume that 
\begin{equation} \label{e_vnaM}
     \cl M = \cl B (H) \bar \otimes \cl N_{1},
\end{equation}
where $ (\cl N_{1}, \tau_{1})$ is a finite tracial von Neumann algebra and 
$ J (x) = x \otimes 1_{\cl N_{1}}$ for $ x\in \cl B(H)$. Then $J_1^*: \cl B(H)\bar\otimes \cl N_1\to \cl B(H)$ is given by $J_1^*=\text{id}\otimes\tau_{1}$. 

Consider $ \cl N_{1} \subseteq \cl B(H_{1})$, so that  
\[\cl M = \cl B(H) \bar \otimes \cl N_{1} \subseteq \cl B( H \otimes_{} H_{1} ) \]
as a von Neumann subalgebra. Use \cite[Proposition IV 1.8]{tak1} again, 
observing that $ \{ U(\epsilon_{i,j}\otimes 1_{\cl N_{1}})\}$ is a matrix unit system in 
$ \cl M$, set $ q_0 = U(\epsilon_{i_0,i_0} \otimes 1_{\cl N_{1}})$, $ H_2 = q_{0}(H \otimes H_{1})$ and $ \cl N_{2} = q_0 \cl M q_0$, and equip $ \cl N_{2}$ with the restriction $ \tau_{2}$ of the trace $ \tau_{\cl M}$. Note that $\tau_{2}(1_{\cl N_{2}}) = 1$.
As before, let 
$\pi : \cl B(H) \bar \otimes \cl N_{1} \rightarrow \cl B(H) \bar \otimes \cl N_{2}$
be a trace-preserving *-isomorphism, such that 
\begin{equation} \label{e_piu}
    \left(\pi \circ U\right)(z \otimes 1_{\cl N_{1}})= z \otimes 1_{\cl N_{2}}.
\end{equation}
We have that $ \cl N_{2} \subseteq \cl B(H_{2})$, so that 
$\cl B(H) \bar \otimes \cl N_{2} \subseteq \cl B(H \otimes_{} H_{2}) $. 
In addition, the *-isomorphism $ \pi$ 
appearing in the proof of 
\cite[Proposition IV 1.8]{tak1} has the form 
\begin{align} \label{e_piconj}
    \pi( Y) = D Y D^*, \quad Y \in \cl B(H) \bar \otimes \cl N_{1},
\end{align}
where $D : H\otimes_{} H_{1} \to H \otimes_{} H_{2}$ is unitary. 
By (\ref{e_piu}) and (\ref{e_piconj}), 
\begin{equation} \label{e_formula} 
    U(z \otimes 1_{\cl N_{1}}) = D^{*}(z\otimes 1_{\cl N_{2}})D, \quad z\in \cl B(H).
\end{equation}
Therefore, 
\begin{equation}\label{eq_Phimod}
\Phi(x) = (\id \otimes \tau_{\cl N_{1}})\big(D^{*}(x \otimes 1_{\cl N_{2}})D \big), 
\ \ \ x\in \cl B(H),
\end{equation}
for the finite tracial von Neumann algebra $(\cl N_{2},\tau_{2})$ and the unitary 
operator
$D \in \cl B(H ) \bar \otimes \cl B(H_{1},H_{2})$
(we view $B(H_{1},H_{2}) $ as a weak* closed subspace of $ \cl B(H_{1} \oplus H_{2})$).  

Let $ v \in (\cl B(H_{1},H_{2}))_{*}$ be a normal functional; thus, $L_{v}(D) \in \cl B(H)$. We will show that 
$L_{v}(D) \in \cl D $; as $ v $ is arbitrary, 
the slice map property \cite{kraus} will imply that $ D \in \cl D \bar \otimes \cl B(H_{1},H_{2})$. 
Let $ u \in \cl D'$ be unitary; it suffices to show that 
$L_{v}(D)u = uL_{v}(D)$. Since
$$L_{v}(D)u = L_{v}\big(D (u \otimes I_{H_{1}})\big) 
\ \mbox{ and } \ 
uL_{v}(D)  = L_{v}\big( (u \otimes I_{H_{2}} ) D \big),$$
it suffices to show that 
\begin{align} \label{e_intertwining}
    (u \otimes I_{H_{2}} ) D =D (u \otimes I_{H_{1}}).
\end{align}
Set
\begin{align*}
D_{u}:= (u \otimes I_{H_{2}} ) D - D (u \otimes I_{H_{1}}),
\end{align*}
and note that 
$$D_{u}^{*}D_{u}= 2 (I_{H} \otimes I_{H_{1}}) -  D^{*}(u^{*}\otimes I_{H_{2}})D(u\otimes I_{H_{1}}) - (u^{*}\otimes I_{H_{1}})D^{*}(u\otimes I_{H_{2}})D.$$   
By (\ref{e_formula}), 
\begin{equation*}
D^{*}(u^{*}\otimes I_{H_{2}})D(u\otimes I_{H_{1}}) = U(u^{*}\otimes I_{H_{1}}) \cdot (u \otimes I_{H_{1}}) \in \cl B(H) \bar \otimes \cl N_{1},
\end{equation*}
and 
\begin{equation*}
(u^{*}\otimes I_{H_{1}})D^{*}(u\otimes I_{H_{2}})D= (u^{*} \otimes I_{H_{1}}) \cdot U(u\otimes I_{H_{1}})  \in \cl B(H) \bar \otimes \cl N_{1}.
\end{equation*}
Let $ \omega \in \cl S_1(H)$; 
then by (\ref{e_formula}) and the modularity of $ \Phi$, 
we have 
\begin{eqnarray}\label{e_tracea}
&&\sca{D^{*}(u^{*}\otimes I_{H_{2}})D(u\otimes I_{H_{1}}), \omega \otimes I_{H_{1}} }_{\cl M, \cl M_{*}} \nonumber\\
&& =\sca{ U(u^{*} \otimes I_{H_{1}}) \cdot (u\otimes I_{H_{1}}), \omega \otimes I_{H_{1}} }_{\cl M, \cl M_{*}} \nonumber\\ 
&&= \sca{ UJ(u^{*} ) \cdot J(u), J_{1}(\omega) }_{\cl M, \cl M_{*}} 
= \tau_{\cl M}\big( UJ(u^{*} ) \cdot J(u) \cdot J_{1}(\omega) \big) \nonumber\\
     &&=\tau_{\cl M}\big( UJ(u^{*}) \cdot J(u) \cdot J(\omega) \big)
=  \tau_{\cl M}\big( UJ(u^{*} ) \cdot J(u\omega) \big) \\
  &&   = \tau_{\cl M}\big( UJ(u^{*} ) \cdot J_{1}(u\omega) \big)
= \sca{ UJ(u^{*} ) , J_{1}(u \omega)  }_{\cl M, \cl M_{*}}\nonumber\\
     &&=  \sca{ J_{1}^{*} UJ(u^{*} ), u\omega  }_{\cl B(H), \cl S_{1}(H)}
 = \sca{ \Phi (u^{*} ), u\omega  }_{\cl B(H), \cl S_{1}(H)} \nonumber\\
 && = {\rm tr}(\Phi(u^{*})u\omega)
 = {\rm tr}(\Phi(u^{*}u)\omega) 
  = \tr(\omega).\nonumber
\end{eqnarray}
Similarly, 
\begin{eqnarray} \label{e_traceb}
    &&\sca{ (u^{*} \otimes I_{H_{1}}) \cdot U(u\otimes I_{H_{1}}), \omega \otimes I_{H_{1}} }_{\cl M, \cl M_{*}}\nonumber\\
&&=\sca{   U(u\otimes I_{H_{1}}), \omega u^{*}\otimes I_{H_{1}} }_{\cl M, \cl M_{*}}
= \sca{   J_{1}^{*}UJ(u), \omega u^{*}}_{\cl B(H), \cl S^{1}(H)}  \\
&&= \sca{   \Phi(u), \omega u^{*}}_{\cl B(H), \cl S^{1}(H)} 
=\tr(\omega). \nonumber
\end{eqnarray} 
Combining (\ref{e_tracea}) and (\ref{e_traceb}), we have that 
\begin{align*}
    \sca{D^{*}_{u}D_{u}, \omega \otimes I_{H_{1}}}_{\cl M,\cl M_{*}} = 0, \quad \omega \in \cl S_{1}(H),
\end{align*}
which in turn implies that, if $ \omega \in \cl S_{1}(H)^{+}$, then
$$    0 = \sca{D^{*}_{u}D_{u}, \omega \otimes I_{H_{1}}}_{\cl M,\cl M_{*}}
    = \sca{ L_{\omega}(D^{*}_{u}D_{u}), I_{H_{1}} }_{\cl N_1,(\cl N_1)_{*}} 
    = \tau_{\cl N_1}\big( L_{\omega}(D^{*}_{u}D_{u}) \big). $$
Finally, the faithfulness of the trace implies that 
\begin{align*}
    L_{\omega}(D^{*}_{u}D_{u})=0, \text{ for all } \omega \in \cl S_{1}(H)^{+},
\end{align*}
which further implies that $D^{*}_{u}D_{u} =0$, as desired.

Let $H_{3}$ be a Hilbert space of sufficiently large dimension, 
so that $ H_{1} \otimes_{} H_{3} \cong H_{2} \otimes_{} H_{3}$ via a unitary $V : H_{1} \otimes_{} H_{3} \to H_{2} \otimes_{} H_{3} $. 
Set 
    \[ D_{1} = ( 1_{\cl D} \otimes V^{*})(D \otimes I_{H_{3}}) \in \cl D \bar \otimes \cl B(H_{1}\otimes_{}H_{3})\] 
and replace the von Neumann algebra $\cl N_{1}$ with $ \cl N= \cl N_{1} \otimes I_{H_{3}} \subseteq \cl B(H_{1} \otimes_{} H_{3})$. 
Letting $ K= H_{1} \otimes_{} H_{3}$, we have that $ \cl N \subseteq \cl B(K)$.
Letting $x = \epsilon_{i,j}$ in (\ref{e_formula}) and setting $d_{i,j} = L_{\epsilon_{j,i}}(D_{1})$, $i,j\in \bb{I}$, 
we obtain 
 \[ d_{i,k}^{*}d_{l,j} \in \cl N_{}, \quad i,j,k,l \in \bb{I}, \]
which implies that $ D_{ 1}$ returns to $\cl N$ 
(see Lemma \ref{l_returns}). 
Thus $ D_{1}^*(x \otimes 1_{\cl N})D_{ 1} \in \cl B(H) \bar \otimes \cl N$ and
therefore,
\begin{equation*}
    \Phi(x) = (\id \otimes \tau_{\cl N})\big( D_{ 1}^*(x \otimes 1_{\cl N})D_{ 1} \big)
\end{equation*}
for a unitary $D_{ 1} \in \cl D \bar \otimes \cl B(K)$,  where $\tau_{\cl N}(a\otimes I_{H_3})=\tau_{1}(a)$, $a\in\cl N_1$. Finally, 
note that, for any $ i \in \bb I$, we have 
\begin{align*}
    D_1^{*}(\epsilon_{i,i}\otimes 1_{\cl N})D_1 = D^*(\epsilon_{i,i}\otimes I_{H_{2}})D \otimes I_{H_{3}}= \pi^{-1}(\epsilon_{i,i}\otimes I_{H_{2}}) \otimes I_{H_{3}}
\end{align*}
which, since $ \pi$ is a trace preserving *-automorphism, implies that 
\begin{align*}
    (\tr \otimes \tau_{\cl N})\big( D_{1}^{*}( \epsilon_{i,i} \otimes I_{\cl N})D_{1} \big)= \tr  \otimes \tau_{{1}}(\pi^{-1}(\epsilon_{i,i} \otimes I_{H_{2}})) = 1;
\end{align*}
hence $D_1$ satisfies trace preservation. 
We have thus proved statement (ii) with $ D_1$ in place of $ D$.

(ii) $\Rightarrow $ (i) 
Note first that since $ D \in \cl D \bar \otimes \cl B(K)$ returns to $ \cl N$ we have that 
$$ D^{*}(x \otimes 1_{\cl N})D \in \cl B(H) \bar \otimes \cl N, $$
 for all $ x\in \cl B(H) $.
We check the $\cl D'$-modularity of $\Phi$; 
let $ a\in \cl D'$ and $ x\in \cl B(H)$.
For (finite) sums 
$d = \sum_{i=1}^k c_{i}\otimes m_{i}$ and $d'= \sum_{j=1}^k c_{j}'\otimes m_{j}'$, with $ c_{i},c_{j}' \in \cl D$ and $ m_{i}, m_{j}' \in \cl B(K)$, $i = 1,\dots,k$, we have
\begin{align*}
d^*(ax\otimes 1_{\cl N})d' 
& = \sum_{i,j=1}^k c_{i}^{*}axc_{j}'\otimes m_{i}^*m_{j}'  
=\sum_{i,j=1}^k ac_{i}^{*}xc_{j}'\otimes m_{i}^*m_{j}' \\
& = (a\otimes 1_{\cl N})\left(\hspace{-0.05cm}\sum_{i,j=1}^k c_{i}^{*}xc_{j}'\otimes m_{i}^*m_{j}'\hspace{-0.1cm}\right) 
\hspace{-0.05cm} = \hspace{-0.05cm} (a\hspace{-0.05cm}\otimes \hspace{-0.05cm} 1_{\cl N} )
\left(\hspace{-0.05cm}d^*(x\otimes 1_{\cl N})d'\hspace{-0.05cm}\right)\hspace{-0.05cm}.
\end{align*}
Approximation in the weak* topology implies
\[  D^*(ax\otimes 1_{\cl N})D = (a\otimes 1_{\cl N} ) \big( D^*(x\otimes 1_{\cl N})D\big). \]
We have
\begin{align*}
    (\id \otimes \tau_{\cl N}) \big( D^*(ax\otimes 1_{\cl N})D \big) 
    & = (\id \otimes \tau_{\cl N})(a\otimes 1_{\cl N} ) \big( D^*(x\otimes 1_{\cl N})D\big)= \\
     & = a (\id \otimes \tau_{\cl N}) \big( D^*(x\otimes 1_{\cl N})D\big) 
\end{align*}
which
implies the left $\cl D'$-modularity of $\Phi$. The right $\cl D'$-modularity 
follows by symmetry. 

Turning to absolute dilatability, we have that the map
\begin{align*}
    \beta : \B(H \otimes K) & \rightarrow \B(H\otimes K), \ 
    z \mapsto D^{*}z D, 
\end{align*}
is a trace preserving unital *-automorphism.
Note that 
\begin{align*}
    \B(H) \otimes 1_{\cl N} \subseteq D\big( \B(H)\bar \otimes \cl N \big) D^*\subseteq \B(H) \bar \otimes \cl \B(K).
\end{align*} 
Since $D$ satisfies trace preservation, there exists a
matrix unit system $\{\epsilon_{i,j}\}_{i,j \in \bb I} \subseteq \B(H)$ 
arising from an orthonormal basis, and an index $i_0\in \bb{I}$, such that 
\begin{equation}\label{eq_i0i0}
(\tr \otimes \tau_{\cl N})\big(D^{*}(\epsilon_{i_0,i_0} \otimes 1_{\cl N})D\big)=1.
\end{equation}
We have that 
$ \{\epsilon_{i,j} \otimes 1_{\cl N} \}_{i,j \in \bb I}$ is a matrix unit system in 
the von Neumann algebra
$\beta^{-1}(\B(H)\bar \otimes \cl N) = D\big( \B(H)\bar \otimes \cl N \big) D^* $;
by 
\cite[Proposition IV 1.8]{tak1}, there exists a von Neumann algebra $M$ such that $\beta^{-1}(\B(H)\bar \otimes \cl N) \cong \B(H)\bar \otimes M$. Set $q:= \epsilon_{i_0,i_0} \otimes 1_{\cl N}$; then 
$ M \cong q  \beta^{-1}(\B(H)\bar \otimes \cl N)q $ and the  *-isomorphism in \cite{tak1}
has the form 
\begin{align*}
    \rho : \beta^{-1}(\B(H)\bar \otimes \cl N) & \rightarrow \cl B(H) \bar \otimes M, 
    \ 
     z \mapsto \sum_{i,j\in \bb{I}} \epsilon_{i,j} \otimes z_{i,j}, 
\end{align*}
where 
\begin{align*}
    z_{i,j}:= (\epsilon_{i_0,i} \otimes 1_{\cl N})z(\epsilon_{j,i_0} \otimes 1_{\cl N}) \in M, \quad i,j \in \bb I.
\end{align*}
Equip $\beta^{-1}(\B(H)\bar \otimes \cl N)$ and $M $ with  $\tau_{\beta}$ and $ \tau_{M}$, respectively, which are the restrictions of $ (\tr\bar \otimes \tau_{\cl N}) \circ \beta$. 
Further, equip $ \cl B(H)\bar \otimes M$ with $ \tr \bar\otimes \tau_{M}$
and note that these traces are all semi-finite.

We observe that $(\tr \otimes \tau_{M})(\rho(z)) = \tau_{\beta}(z)$, for all $ z \in \beta^{-1}(\B(H)\bar \otimes \cl N)$. Indeed,
\begin{align*}
(\tr \otimes \tau_{M})(\rho(z)) 
&= \sum_{i\in \bb I} \tau_{ M}(z_{i,i})\\
& =\sum_{i\in \bb I} \tau_{ \beta}((\epsilon_{i_0,i} \otimes 1_{\cl N})z(\epsilon_{i,i_0} \otimes 1_{\cl N}) \\
& =\sum_{i \in \bb I} \tau_{ \beta}((\epsilon_{i,i_0} \otimes 1_{\cl N})(\epsilon_{i_0,i} \otimes 1_{\cl N})z)\\
&= \sum_{i \in \bb I} \tau_{ \beta}((\epsilon_{i,i} \otimes 1_{\cl N})z)
=  \tau_{\beta}\left(\sum_{i\in \bb I}(\epsilon_{i,i} \otimes 1_{\cl N})z\right)
= \tau_{ \beta}(z).
\end{align*}
Moreover,  $\rho (\epsilon_{i,j} \otimes 1_{\cl N})=  \epsilon_{i,j}\otimes q$ for all $ i,j \in \bb I$ and thus $ \rho(x \otimes 1_{\cl N}) = x \otimes q$ for all $ x\in \B(H)$.
In addition, by (\ref{eq_i0i0}), 
\begin{align*}
    \tau_{M}(q) = \tau_{ M}(\epsilon_{i_0,i_0} \otimes 1_{\cl N}) = 1,
\end{align*}
which implies that $(M, \tau_{M})$ is a finite von Neumann algebra. 
Noting that the identity $ 1_{M}$ of $M$ can be 
identified with the projection $q$, for all $ z  \in \beta^{-1}(\cl B(H)\bar\otimes\cl N)$, we have
\begin{align*}
    (\tr \otimes \tau_{\cl N})(D^{*}z D) = \tau_{\beta}(z) = 
    (\tr\otimes\tau_{M})(\rho(z)).
\end{align*} 
The isomorphism $\rho:\beta^{-1}(\cl B(H)\bar\otimes\cl N)\to\cl B(H)\bar\otimes M$ is implemented by a unitary $W$, 
that is, $\rho(z)=W^*zW$
(see the proof of \cite[Proposition IV 1.8]{tak1}). As $\rho(x\otimes 1_{\cl N})=x\otimes 1_M$, $x\in \cl B(H)$, we obtain that $W$ is of the form $W=I_{\cl B(H)}\otimes V$, for a unitary $V$. Let $\tilde D=W^*D$. We obtain 
the trace-preserving *-isomorphism 
\begin{align*}
    \gamma_{} : \B(H)\bar \otimes M \to \B(H) \bar \otimes \cl N, \ 
                       z \mapsto {\tilde D}^* z\tilde D, 
\end{align*}
that has the property 
${\tilde D}^*(x\otimes 1_M)\tilde D=D^*(x\otimes 1_{\cl N})D$, $x\in\cl B(H)$.  

We now proceed as in \cite[Theorem 7.1]{dlem}. Set 
\begin{align*}
    M_{-}= \bar \otimes_{k\leq -1}M \quad \text{and} \quad \cl N_{+}= \bar \otimes_{k \geq 1} \cl N, 
\end{align*}
and define 
\begin{align*}
    \cl M = \B(H) \bar \otimes M_{-} \bar \otimes M \bar \otimes \cl N \bar \otimes \cl N_{+},
\end{align*}
 equipped with the product trace. Set 
\begin{align*}
    \cl M_{1}=  M_{-} \bar \otimes \big (\B(H) \bar \otimes M \big ) \bar \otimes \cl N \bar \otimes \cl N_{+}  \quad \text{and} \quad \cl M_{2}=  M_{-} \bar \otimes M \bar \otimes \big( \B(H) \bar \otimes  \cl N \big) \bar \otimes \cl N_{+}
\end{align*}
and let 
$ \kappa_{1} : \cl M \to \cl M_1$ and $ \kappa_{2} : \cl M \to \cl M_{2}$ 
be the *-isomorphisms, arising from the corresponding flips. 
Further, consider the shifts 
\begin{align*}
    \sigma_{-} : M_{-} \to M_{-}\bar \otimes M, \ 
                   \otimes_{k \leq -1}x_{k} \mapsto ( \otimes_{k \leq -1}x_{k-1}) \otimes x_{-1}
\end{align*}
and 
\begin{align*}
    \sigma_{+} : \cl N \bar \otimes \cl N_{+} \to \cl N_{+}, \ 
                       x_{0} \otimes ( \otimes_{k \geq 1}x_{k})  \mapsto  \otimes_{k \geq 1}x_{k-1}
\end{align*}
which are trace-preserving *-isomorphisms. Set 
\begin{align*}
     U : \cl M \to \cl M; \quad U= \kappa_{2}^{-1} \circ ( \sigma_{-} \bar \otimes \gamma \bar \otimes \sigma_{+}) \circ \kappa_{1}, 
\end{align*}
and note that $U$ is a trace preserving *-automorphism. Finally, let 
 \begin{align*}
     J: \B(H) \to \cl M, \ 
           x \mapsto x \otimes 1,
 \end{align*}
where $ 1$ denotes the unit of $ M_{-} \bar \otimes M \bar \otimes \cl N \bar \otimes \cl N_{+}$.
Our goal is to show that, for all $ x\in \cl B(H)$ and $ y \in \cl S_{1}(H)$, we have that 
\begin{equation} \label{e_casen1}
    \sca{J_{1}^{*} U^{n} J( x), y }_{ \cl B(H), \cl S_{1}(H)} = \sca{ \Phi^{n}(x), y }_{\cl B(H), \cl S_{1}(H)}.
\end{equation}

Consider first the case $n=1$. The right-hand side of (\ref{e_casen1}) then becomes
\begin{multline*}
    \sca{ \Phi^{}(x), y }_{\cl B(H), \cl S_{1}(H)} 
    = \sca{(\id_{} \otimes\tau_{\cl N})\big(D^*(x\otimes 1_{\cl N})D\big) , y }_{\cl B(H), \cl S_{1}(H)}\\
    = \tr\big( (\id_{} \otimes\tau_{\cl N})(D^*(x\otimes 1_{\cl N})D) \cdot y \big).
\end{multline*}
On the other hand, 
\begin{align} \label{e_induct1}
    UJ(x) & = U(x \otimes 1)= \nonumber \\
          & =  \kappa_{2}^{-1} \circ ( \sigma_{-} \bar \otimes \gamma \bar \otimes \sigma_{+}) \big( 1_{M_{-}} \otimes ( x \otimes 1_{M}) \otimes 1_{\cl N} \otimes 1_{\cl N_{+}} \big)= \nonumber \\
          & =  \kappa_{2}^{-1} \big( 1_{M_{-}} \otimes 1_{M} \otimes \big(  \tilde D^*( x \otimes 1_{M})\tilde D \big) \otimes 1_{\cl N_{+}} \big).
\end{align}
Let 
$\left(\sum_{i=1}^{n_{\alpha}} x_{\alpha,i}\otimes y_{\alpha,i}\right)_{\alpha\in \bb{A}}$,  where $x_{\alpha,i}\in\cl B(H)$ and $y_{\alpha,i}\in \cl N$, be a net, 
converging to $\tilde D^*(x\otimes 1_{M})\tilde D$ in weak* topology.
Using the weak* continuity of $ \kappa_{2}^{-1}$, we obtain 
(all the limits below being in the weak* topology)
\begin{align*}
    UJ(x)   & =  \kappa_{2}^{-1} \left(1_{M_{-}} \otimes 1_{M} \otimes 
    \left(\lim_{\alpha\in \bb{A}} \sum_{i=1}^{n_{\alpha}} x_{\alpha, i}\otimes y_{\alpha,i}\right) \otimes 1_{\cl N_{+}} \right) \\
& =  \lim_{\alpha\in \bb{A}} \sum_{i=1}^{n_{\alpha}} 
\kappa_{2}^{-1} \left( 1_{M_{-}} \otimes 1_{M} \otimes 
\left(x_{\alpha,i}\otimes y_{\alpha,i} \right) \otimes 1_{\cl N_{+}} \right) \\
& =  \lim_{\alpha\in \bb{A}} \sum_{i=1}^{n_{\alpha}}  
    \left( x_{\alpha,i} \otimes  1_{M_{-}} \otimes 1_{M} \otimes 
    y_{\alpha,i}  \otimes 1_{\cl N_{+}} \right). 
\end{align*}
Hence 
 \begin{eqnarray*}
&&\sca{J_{1}^{*} U^{} J( x), y }_{\cl \cl B(H), \cl S_{1}(H)} 
=\sca{ U^{} J( x), J_{1}(y )}_{\cl M, \cl M_{*}}\\
&& = \lim_{\alpha\in \bb{A}} \sum_{i=1}^{n_{\alpha}} 
\left\langle\left(x_{\alpha,i} \otimes  1_{M_{-}} \otimes 1_{M} 
\otimes y_{\alpha,i}  \otimes 1_{\cl N_{+}} \right)\hspace{-0.1cm}, 
y \hspace{-0.05cm}\otimes \hspace{-0.05cm} 1_{M_{-}} 
\hspace{-0.05cm} \otimes \hspace{-0.05cm} 1_{M} \hspace{-0.05cm} 
\otimes \hspace{-0.05cm} 1_{\cl N} \hspace{-0.05cm}  \otimes 
\hspace{-0.05cm} 1_{\cl N_{+}}\right\rangle
\\
&& = \lim_{\alpha\in \bb{A}} \sum_{i=1}^{n_{\alpha}} 
\tr(x_{\alpha,i} y) \tau_{\cl N}(y_{\alpha,i})
= \lim_{\alpha\in \bb{A}} \sum_{i=1}^{n_{\alpha}} 
(\tr \otimes \tau_{\cl N})
\left((x_{\alpha,i} \otimes y_{\alpha,i})(y \otimes 1_{\cl N})\right)\\
&& = (\tr \otimes \tau_{\cl N}) 
\left(   \tilde D^*(x\otimes 1_{M})\tilde D \cdot  (y \otimes 1_{\cl N}) \right). 
\end{eqnarray*}
Identity (\ref{e_casen1}) will therefore follow
from the fact that
\begin{equation}\label{eq_zyin}
\tr\left( (\id_{}\otimes\tau_{\cl N})(z) \cdot y \right) 
=(\tr \otimes \tau_{\cl N^{}}) \left( z\cdot (y\otimes 1_{\cl N}) \right), \ \ z\in \cl B(H)\bar \otimes \cl N,
\end{equation} 
which is straightforward to verify.

We next show that, if $n\in \bb{N}$ and $x\in \cl B(H)$, then 
\begin{eqnarray}\label{eq_Unind}
&& 
U^{n}J(x) = \\ 
&& 
\kappa_{2}^{-1}\Big( 1_{M_{-}} 
\hspace{-0.05cm}\otimes \hspace{-0.05cm} 1_{M} 
\hspace{-0.05cm}\otimes \hspace{-0.05cm}
\left( \tilde D_{1,2}^*\ldots \tilde D_{1,n+1}^*
\left(x 
\hspace{-0.05cm} \otimes \hspace{-0.05cm}
1_{M}^{(n)}\right)
\tilde D_{1,n+1}\ldots \tilde D_{1,2} \right) 
\hspace{-0.05cm}\otimes \hspace{-0.05cm}
\underset{\mathclap{\substack{\uparrow \\ n }}}{1_{\cl N}} \hspace{-0.05cm} \otimes\hspace{-0.05cm}\cdots \Big), \nonumber
\end{eqnarray}
where we use leg-notation $\tilde D_{1,k}$ to denote the operator 
acting as $\tilde D$ on the first and $k$-th component of the tensor product. 
Here $ 1_{M}^{(k)}$ denotes the element 
    $ \underbrace{1_{M} \otimes \cdots \otimes 1_{M}}_{k \ \mbox{\tiny times}}$. 
The case $n = 1$ was already shown in (\ref{e_induct1}).
We use induction; assuming that (\ref{eq_Unind}) holds,
let $x_{\alpha, i}\in\cl B(H)$, $y_{\alpha,i}\in \cl N^{\otimes n}$, 
$\alpha\in \bb{A}$, $i = 1,\dots,n_{\alpha}$, be such that 
$$\tilde D_{1,2}^*\ldots \tilde D_{1,n+1}^*\left(x 
\hspace{-0.05cm} \otimes \hspace{-0.05cm}
1_{M}^{(n)}\right)\tilde D_{1,n+1}\ldots \tilde D_{1,2}
= \lim_{\alpha\in \bb{A}} \sum_{i=1}^{n_{\alpha}} 
x_{\alpha, i} \hspace{-0.05cm} \otimes \hspace{-0.05cm} y_{\alpha, i}$$
in weak* topology. 
The weak* continuity of the flip $M\bar\otimes\cl B(H)\to 
\cl B(H)\bar\otimes M$ implies that 
$$\tilde D_{1,3}^*\ldots \tilde D_{1,n+2}^*\left(x 
\hspace{-0.05cm} \otimes \hspace{-0.05cm}
1_{M}^{(n+1)}\right)\tilde D_{1,n+2}\ldots \tilde D_{1,3}
= \lim_{\alpha\in \bb{A}} \sum_{i=1}^{n_{\alpha}} 
x_{\alpha, i} \hspace{-0.05cm} \otimes \hspace{-0.05cm}
1_{M} \hspace{-0.05cm} \otimes \hspace{-0.05cm} y_{\alpha, i}$$
in weak* topology, and hence
\begin{align*}
& U^{n+1}J(x)  = U\left(\lim_{\alpha\in \bb{A}} \sum_{i=1}^{n_{\alpha}}  
x_{\alpha, i} \otimes  1_{M_{-}} \otimes 1_{M} \otimes y_{\alpha,i} \otimes 1_{\cl N_{+}} \right) \\
&=  \lim_{\alpha\in \bb{A}} \sum_{i=1}^{n_{\alpha}}  
U\left( x_{\alpha,i} \otimes  1_{M_{-}} \otimes 1_{M} \otimes y_{\alpha,i} 
\otimes 1_{\cl N_{+}} \right)\\
&= \lim_{\alpha\in \bb{A}} \sum_{i=1}^{n_{\alpha}} 
\kappa_{2}^{-1} \circ (\sigma_{-}\bar \otimes \gamma \bar \otimes \sigma_{+})\Big( 1_{M_{-}} \otimes (x_{\alpha, i}\otimes 1_{M}) \otimes y_{\alpha,i} \otimes 1_{\cl N_{+}} \Big)\\
&= \lim_{\alpha\in \bb{A}} \sum_{i=1}^{n_{\alpha}} 
\kappa_{2}^{-1} \Big(1_{M_{-}} \otimes 1_{M} \otimes 
\big(\tilde  D_{}^{*}(x_{\alpha,i} \otimes 1_{M})\tilde D \big)\otimes y_{\alpha,i} \otimes \underset{\mathclap{\substack{\uparrow \\ n }}}{1_{\cl N}}\otimes \cdots \Big) \\
&= \lim_{\alpha\in \bb{A}} \sum_{i=1}^{n_{\alpha}} 
\kappa_{2}^{-1} \Big(1_{M_{-}} \hspace{-0.05cm}\otimes\hspace{-0.05cm} 1_{M} 
\hspace{-0.05cm}\otimes \hspace{-0.05cm}
\big( \tilde D_{1,2}^{*}(x_{\alpha,i} 
\hspace{-0.05cm}\otimes \hspace{-0.05cm} 1_{M}\otimes\hspace{-0.05cm} y_{\alpha,i})
\tilde D_{1,2} \big) 
\hspace{-0.05cm}\otimes\hspace{-0.05cm} \underset{\mathclap{\substack{\uparrow \\ n+1 }}}{1_{\cl N}}
\hspace{-0.05cm}\otimes \hspace{-0.05cm}\cdots \Big) \\
& = \kappa_{2}^{-1} \Big(1_{M_{-}} 
\hspace{-0.1cm} \otimes \hspace{-0.1cm}  1_{M} \hspace{-0.1cm} \otimes \hspace{-0.1cm} 
\big( \tilde D_{1,2}^{*}\dots \tilde D_{1,n+2}^{*}(x_{} 
\hspace{-0.1cm} \otimes \hspace{-0.1cm} 1_{M}^{(n+1)})
\tilde D_{1,n+2}\dots \tilde D_{1,2} \big) 
\hspace{-0.1cm} \otimes \hspace{-0.1cm} \underset{\mathclap{\substack{\uparrow \\ n+1 }}}{1_{\cl N}}
\hspace{-0.1cm} \otimes \hspace{-0.05cm}  \cdots \Big).
\end{align*}
We claim that, on the other hand,
\begin{equation}\label{eq_Phiin}
\Phi^n(x)=(\id\otimes\underbrace{\tau_{\cl N}\otimes\ldots\otimes\tau_{\cl N}}_n)(\tilde D_{1,2}^*\cdots \tilde D_{1,n+1}^*(x\otimes 1_{M}^{(n)})\tilde D_{1,n+1}\cdots \tilde D_{1,2}).
\end{equation}
Using induction again, we note that the formula
was established already for $n = 1$, and assume (\ref{eq_Phiin}).
Then 
\begin{eqnarray*}
&&\Phi^{n+1}(x)
= 
\Phi((\id\otimes\tau_{\cl N}^{\otimes n})(\tilde D_{1,2}^*\cdots \tilde D_{1,n+1}^*(x\otimes 1_{M}^{(n)})\tilde D_{1,n+1}\cdots \tilde D_{1,2})\\
&&=
\lim_{\alpha\in \bb{A}} \sum_{i=1}^{n_{\alpha}}
\Phi((\id\otimes\tau_{\cl N}^{\otimes n})(x_{\alpha,i}\otimes y_{\alpha,i}))
=
\lim_{\alpha\in \bb{A}} \sum_{i=1}^{n_{\alpha}}
\Phi(x_{\alpha,i} \tau_{\cl N}^{\otimes n}(y_{\alpha,i}))\\
&&=
\lim_{\alpha\in \bb{A}} \sum_{i=1}^{n_{\alpha}}
(\id\otimes\tau_{\cl N})(D^*(x_{\alpha,i}\otimes 1_{\cl N})D) 
\hspace{0.05cm}\tau_{\cl N}^{\otimes n}(y_{\alpha,i}))\\
&&=
\lim_{\alpha\in \bb{A}} \sum_{i=1}^{n_{\alpha}}
\left(\id\otimes\tau_{\cl N}^{\otimes n+1}\right)
\left(D_{1,2}^*(x_{\alpha, i}\otimes 1_{\cl N}\otimes y_{\alpha, i})D_{1,2}\right)\\
&&=
\left(\id\otimes\tau_{\cl N}^{\otimes n+1}\right)
\left(\tilde D_{1,2}^*\cdots \tilde D_{1,n+2}^*(x\otimes 1_{M}^{(n+1)})\tilde D_{1,n+2}\cdots \tilde D_{1,2}\right).
\end{eqnarray*}
Now, an identity similar to (\ref{eq_zyin})  and expressions (\ref{eq_Unind}) and (\ref{eq_Phiin}) imply 
$$\langle J_1^*U^nJ(x),y\rangle_{\cl B(H),\cl S_1(H)}
= 
    \langle U^n(x\otimes 1),y\otimes 1 \rangle_{\cl M,\cl M_*}
= 
\langle \Phi^n(x),y\rangle_{\cl B(H),\cl S_{1}(H)},$$
showing that $\Phi$ is absolutely dilatable.

(ii)$\Leftrightarrow$(iii) 
Assume (ii) holds. 
Using Lemma \ref{l_symbolc} and (\ref{eq}), we have that 
\begin{eqnarray*}
& &\langle u_{\Phi_D} (e_i\otimes e_k), e_l\otimes e_j\rangle
=  
\left\langle\Phi_D(\epsilon_{i,j}),\epsilon_{k,l}\right\rangle\\
& &=  
\left\langle (\id\otimes\tau_{\cl N})(D^*(\epsilon_{i,j}\otimes 1_{\cl N})D),\epsilon_{k,l}\right\rangle\\
 &&=\tau_{\cl N}(L_{\epsilon_{k,l}}(D^*(\epsilon_{i,j}\otimes 1_{\cl N})D))
= \tau_{\cl N}(d_{i,l}^*d_{j,k})\\
& &= 
\langle (\id\otimes\id\otimes \tau_{\cl N})(D_{1,3}^* D_{2,3})(e_i\otimes e_k), e_l\otimes e_j \rangle
\end{eqnarray*} 
and hence $u_{\Phi_D}=(\id\otimes\id\otimes \tau_{\cl N})(D_{1,3}^*D_{2,3})$. Reversing the arguments we obtain that if $u_\Phi$ is of the form as in (iii) then $$\langle\Phi(\epsilon_{i,j}),\epsilon_{k,l}\rangle=\langle \id\otimes \tau_{\cl N})(D^*(\epsilon_{i,j}\otimes 1_{\cl N})D),  \epsilon_{k,l}\rangle.$$ 
The fact that $\Phi=\Phi_D$ follows from linearity and weak*-continuity. \end{proof}

\begin{remark}\label{r_ancilla}
\rm 
It follows from the proof of Theorem \ref{tar1} that a 
unital trace-preserving $\cl D'$-bimodule map $\Phi : \cl B(H)\to \cl B(H)$
is absolutely dilatable if and only if there exists a 
finite von Neumann algebra $(\cl N,\tau_{\cl N})$ and 
a normal trace-preserving 
*-automorphism $U : \cl B(H)\bar\otimes\cl N\to \cl B(H)\bar\otimes\cl N$
such that, if 
$J : \cl B(H)\to \cl B(H)\bar\otimes\cl N$ and 
$J_1 : \cl S_1(H)\to \cl S_1(H)\hat{\otimes} \cl N_*$ are the ampliation maps $x \mapsto x\otimes 1$, then 
$\Phi = J_1^*\circ U \circ J$. 
In the sequel, we drop the adjective \lq\lq absolutely'', and simply refer 
to $\Phi$ being dilatable, and we call the von Neumann algebra $\cl N$
its \emph{ancilla}. 
\end{remark}

\begin{remark} \label{r_sepancilla}
  \rm  
  The proof of Theorem \ref{tar1} shows that,
  if we assume that the map $ \Phi$ in its statement 
  admits a separable absolute dilation, then the ancilla $\cl N$
can be chosen to have separable predual.
\end{remark}

\begin{remark}\label{r_findimcase}
\rm 
In \cite[Theorem 2.2]{hm1}, the authors show that 
the factorisable channels $\Phi : M_n\to M_n$ have the form 
$\Phi(x) = \Phi_D$, where $D \in M_n \otimes \cl N$ is a unitary, 
for some finite tracial von Neumann algebra $(\cl N,\tau_{\cl N})$. 
This result can be deduced
from the proof of Theorem \ref{tar1}.  Indeed, in 
this case, $\cl D = M_n$ and 
one can chose $ D$ to be in the von Neumann $M_{n} \otimes \cl N$ in (ii) as follows: The proof remaining unchanged until equation (\ref{e_vnaM}), we note that 
$ \{\epsilon_{i,j}\otimes 1_{\cl N_{1}} \}_{i,j=1}^{n}$ and $ \{U(\epsilon_{i,j}\otimes 1_{\cl N_{1}})\}_{i,j=1}^{n}$ are two systems of matrix units in the von Neumann algebra $\cl M$. By Lemma 2.1 in \cite{hm1}, there exists a unitary $ D $ in $\cl M = M_{n} \otimes \cl N_{1}$  such that 
    \[U(\epsilon_{i,j} \otimes 1_{\cl N_{1}}) 
    = D^{*}(\epsilon_{i,j} \otimes 1_{\cl N_{1}})D, \quad i,j=1,\dots,n. \]
From this we obtain that 
\[U(x \otimes 1_{\cl N_{1}}) = D^{*}(x \otimes 1_{\cl N_{1}})D, \quad x\in M_{n} \]
and the rest of our proof works unchanged. 
\end{remark}

We now discuss how \cite[Theorem 1.1]{dlem}, 
characterising absolutely dilatable Schur multipliers, 
fits into the framework of Theorem \ref{tar1}.
Let $(X,\mu)$ be a $\sigma$-finite measure space and write 
$L^2(X) = L^2(X,\mu)$. Assume that $(X,\mu)$ is separable, meaning that the Hilbert space $L^2(X)$ is separable. Recall that, for any $f\in L^2(X\times X)$ (with respect to the product measure), we may associate an integral operator $T_f\in \cl S_2(L^2(X))$, given by 
$$(T_f\xi)(y)=\int_Xf(x,y)\xi(x)d\mu(x), \xi\in L^2(X),$$
and the mapping $f\mapsto T_f$ is a unitary, 
yielding a Hilbert space identification $L^2(X\times X)\simeq \cl S_2(L^2(X)).$
To any $\phi\in L^\infty(X\times X)$, we associate a bounded linear map $$S_\phi: \cl S_2(L^2(X))\to \cl S_2(L^2(X)),$$
given by $S_\phi(T_f)=T_{\phi f}$. We have
$$\|S_\phi:\cl S_2\to\cl S_2\|=\|\phi\|_\infty.$$
We say that $\phi$ is a {\it Schur multiplier}  if $S_\phi$ is bounded with respect to the operator norm on $\cl S_2(L^2(X))$, 
that is, there exists $C > 0$ such that 
$\|S_\phi(T_f)\|\leq C\|T_f\|$ for every $f\in L^2(X\times X)$. In this case $S_\phi$  can be extended by continuity to a bounded map on the space of compact operators and to a weak$^*$ continuous map on $\cl B(L^2(X))$. We keep the notation $S_\phi$ to denote the latter extension. We note that $S_\phi$ is a $L^\infty(X)$-bimodule map, where we identify $L^\infty(X)$ with the 
corresponding multiplication operator algebra. 
As $L^\infty(X)$ is a maximal abelian von Neumann algebra, 
the map $S_\phi$ is automatically completely bounded on $\cl B(L^2(X))$ \cite{smith-aut}. 
Moreover, $S_\phi$ is positive if and only if it is completely positive; in this case, there exist $a_i\in L^\infty(X)$, $i\in \bb{N}$, 
such that $\text{esssup}\sum_{i=1}^{\infty}|a_i(x)|^2<\infty$ and  the symbol 
$u_{S_\phi} \in L^{\infty}(X\times X)$ 
of $S_\phi$ is given by $u_{S_\phi} = \sum_{i=1}^\infty \bar{a_i}\otimes a_i$. 
It can be easily checked that the series 
converges almost everywhere and hence
we have, in particular, that $$\phi(x,y) =\sum_{i=1}^\infty \overline{a_i(x)}a_i(y)\text{ for a.e. } (x,y)\in X\times X. $$
We refer the reader for the details about Schur multipliers to the survey \cite{tt-survey}.

If $\cl M$ is a von Neumann algebra with a separable predual $\cl M_*$, one calls a function $d:X\to \cl M$ $w^*$-measurable if for all $F\in\cl M_*$, the scalar function $x\mapsto \langle d(x),F\rangle$ is measurable on $X$. Let $L_\sigma(X,\cl M)$ denote the space of all $w^*$-measurable functions $d:X\to \cl M$ defined almost everywhere such that $\|d(\cdot)\|$ is essentially bounded. The natural embedding of  $L^\infty(X)\otimes \cl M$ into $L^\infty_\sigma(X,\cl M)$  extends to a von Neumann algebra identification $L^\infty(X)\bar\otimes \cl M=L^\infty_\sigma(X,\cl M) $
(see \cite[Theorem 1.22.13]{sakai}).

\begin{corollary}
\cite[Theorem 1.1]{dlem}
\label{Schur_mult}
Assume that $(X,\mu)$ is separable and let $\phi\in L^\infty(X\times X)$. The following are equivalent:
\begin{itemize}
\item[(i)] 
the function $\phi$ is a Schur multiplier and $S_\phi: \cl B(L^2(X))\to\cl B(L^2(X))$ admits a separable absolute dilation; 

\item[(ii)] there exist a 
finite tracial von Neumann algebra $(\cl N,\tau_{\cl N})$ with a separable predual and a unitary $d\in L^\infty_\sigma(X,\cl N)$ such that 
    \begin{equation}
        \phi(x,y)=\tau_{\cl N}(d(x)^*d(y)) \text{ for a.e. } (x,y)\in X\times X.
    \end{equation}
\end{itemize}
\end{corollary}
\begin{proof}
    (i)$\Rightarrow$ (ii) 
As $S_\phi$ is $L^\infty(X)$-modular, by Theorem \ref{tar1} there exist a separable Hilbert space $K$,
a tracial von Neumann algebra $(\cl N,\tau_{\cl N})$ acting on $K$, and  a unitary operator $D\in L^\infty(X)\bar\otimes \cl B(H)\equiv L^\infty_\sigma(X,\cl B(H))$ that returns to $\cl N$  such that
    $S_{\phi}=(\id\otimes\tau_{\cl N})\circ\Phi_{D}$ and with the symbol $u_{S_\phi}=(\id\otimes\id\otimes\tau_{\cl N})(D_{2,3}^*D_{1,3})$. Note that $D_{2,3}^*D_{1,3}\in L^\infty(X)\bar\otimes L^\infty(X)\bar\otimes\cl N$ and through natural identification, we obtain that $D^*(x)D(y)\in \cl N$  and $D^*(x)D(x)=1_K$ for almost all $x,y\in X$ (see arguments in the proof of \cite[Lemma 5.2, 5.3]{dlem}). Hence there exists $x_0\in X$ such that $D^*(x_0)D(y)\in\cl N$ for almost all $y\in X$ and $D(x_0)$ is unitary. 
    Consider $d(y)=D^*(x_0)D(y)$. Clearly,
    $$d^*(x)d(y)=D^*(x)D(x_0)D(x_0)^*D(y)=D^*(x)D(y)=(D_{2,3}^*D_{1,3})(x,y)$$ and
    $$\phi(x,y)=u_{S_\phi}(x,y)=\tau_{\cl N}(d^*(x)d(y)) \text{ a.e. }$$

(ii)$\Rightarrow$ (i) 
follows by reversing the arguments in the previous 
paragraph. 
\end{proof}


\section{A hierarchy for dilatable maps}\label{s_hier}

In this section, we define several classes of dilatable 
unital completely positive maps, following an established by now route of differentiating between various ancilla types 
\cite{hm1, hm2, pr}. 
Throughout the section, we assume that all 
absolutely dilatable maps 
admits a separable absolute dilation and, 
by virtue of Remark \ref{r_sepancilla}, the ancillas are chosen to have separable predual. 

Let $I$ be a set, $(\cl M_i, \tau_i)$ be a finite tracial von Neumann algebra, $i\in I$, 
and set $\cl M = \oplus_{i\in I}^\infty \cl M_i$, an $\ell^\infty$-direct sum. Let $\frak{u}$ be a free ultrafilter on the set $I$ and set
\begin{equation}\label{eq_Ju}
\cl J_{\frak{u}} = \left\{(x_i)_{i\in I} 
\in \oplus^{\infty}_{i\in I} \cl M_i : 
\lim\hspace{-0.05cm}\mbox{}_{\frak{u}} \tau_i(x_i^*x_i) = 0\right\},
\end{equation}
where $\lim_{\frak{u}}$ denotes the limit along $\frak{u}$;
clearly, $\cl J_{\frak{u}}$ is a closed two-sided ideal 
of $\oplus^{\infty}_{i\in I} \cl M_i$.

We view $\cl M_i$ in its standard form, acting on the Hilbert space 
$K_i = L^2(\cl M_i,\tau_i)$ arising from the GNS construction applied to 
$\tau_i$, $i\in I$.
We recall \cite[Section 11.5]{Pi2} that the ultraproduct 
$\cl M^{\frak{u}}$ acts on the Hilbert space $K_{\frak{u}}$ of the 
GNS representation of $\cl M$, arising from the state 
$f_{\frak{u}} : \cl M\to \bb{C}$, given by 
\begin{equation}\label{eq_fu}
f_{\frak{u}} \left((x_i)_{i\in I}\right) = 
\lim\hspace{-0.05cm}\mbox{}_{\frak{u}} \tau_i(x_i), \ \ \ 
(x_i)_{i\in I}\in \cl M.
\end{equation}
The GNS representation $\pi_{f_{\frak{u}}} : \cl M\to \cl B(K_{\frak{u}})$
annihilates $\cl J_{\frak{u}}$ and gives rise to a 
faithful *-representation 
$\pi_{\frak{u}} : \cl M/\cl J_{\frak{u}} \to \cl B(K_{\frak{u}})$,
and the ultraproduct of the family 
$((\cl M_i,\tau_i))_{i\in I}$
is defined to be the image $\pi_{\frak{u}}\left(\cl M/\cl J_{\frak{u}}\right)$
of $\cl M/\cl J_{\frak{u}}$ inside $\cl B(K_{\frak{u}})$
(see \cite[Theorem 11.26]{Pi2}). 
We note that $\cl M/\cl J_{\frak{u}}$ 
can be naturally considered as a dense subspace of 
$K_{\frak{u}}$. 
We denote by 
$\tau_{\frak{u}}$ the trace on $\cl M^{\frak{u}}$, induced by the 
functional $f_{\frak{u}}$.
In the case where $\cl M_k = M_{n_k}$ for some $n_k\in \bb{N}$, and $\tau_k$ is the 
normalised trace on $M_{n_k}$, we refer to $\cl M^{\frak{u}}$ as a 
\emph{matricial ultraproduct} (see \cite[Remark 11.32]{Pi2}). 
A tracial von Neumann algebra $(\cl A,\tau_{\cl A})$ is said to embed in 
a matricial ultraproduct $\cl M^{\frak{u}}$, if 
there exists a normal trace-preserving *-monomorphism from $\cl A$ into $\cl M^{\frak{u}}$. 


Let $H$ be a Hilbert space and $\cl D\subseteq \cl B(H)$ be a von Neumann algebra. 
We will call a dilatable $\cl D'$-bimodule map $\Phi : \cl B(H)\to \cl B(H)$
\begin{itemize} 
\item[(i)] \emph{locally factorisable} if it admits an abelian ancilla; 

\item[(ii)] \emph{quantum factorisable} if it admits a finite dimensional ancilla;

\item[(iii)] \emph{approximately quantum factorisable} if it admits an ancilla that 
can be embedded in a matricial ultraproduct $\cl M^{\frak{u}}$.
\end{itemize}
In the sequel, dilatable maps will be also referred to as 
\emph{quantum commuting factorisable}. 
We write $\frak{D}_{{\rm qc},\cl D}(H)$
(resp. $\frak{D}_{{\rm qa},\cl D}(H)$, $\frak{D}_{{\rm q},\cl D}(H)$, 
$\frak{D}_{{\rm loc},\cl D}(H)$) 
for the sets of all factorisable (resp. approximately quantum factorisable, 
quantum factorisable, locally factorisable) $\cl D'$-modular maps on $\cl B(H)$. 
For ${\rm t}\in \{{\rm loc}, {\rm q}, {\rm qa}, {\rm qc}\}$, set 
$\frak{D}_{\rm t}(H) = \frak{D}_{{\rm t},{\cl B(H)}}(H)$. 
If, for a map $\Phi \in \frak{D}_{{\rm t},\cl D}(H)$, 
the unitary operator $D$ in the representation of $\Phi$ in 
Theorem \ref{tar1} (ii) can be chosen from $\cl D\bar\otimes\cl N$, we 
say that $\Phi$ admits a 
\emph{${\rm t}$-exact factorisation}.  
Let 
$\frak{D}_{{\rm t},\cl D}^{\rm ex}(H)$ be the set of all maps 
that admit a ${\rm t}$-exact factorisation.

\begin{proposition}\label{p_convex}
For ${\rm t}\in \{{\rm loc}, {\rm q}, {\rm qa}, {\rm qc}\}$, the sets 
$\frak{D}_{{\rm t},\cl D}(H)$ and 
$\frak{D}_{{\rm t},\cl D}^{\rm ex}(H)$ are convex. 
\end{proposition}

\begin{proof}
First consider the case ${\rm t} = {\rm qc}$. 
Suppose that $(\cl N_i, \tau_i)$ is a finite von Neumann algebra
that is an ancilla for the dilatable map $\Phi_i$, $i = 1,2$, and 
$\Phi = \lambda_1 \Phi_1 + \lambda_2\Phi_2$ as a convex combination. 
We equip $\cl N := \cl N_1\oplus \cl N_2$ with the tracial state 
$\tau$, given by $\tau((z_1,z_2)) = \lambda_1 \tau_1(z_1) + \lambda_2\tau_2(z_2)$. 
Letting $U_i : \cl B(H)\bar\otimes\cl N_i$ be a normal trace-preserving 
*-automorphism, we have that the map
$U : \cl B(H)\bar\otimes\cl N\to \cl B(H)\bar\otimes\cl N$, 
given by 
$U(x\otimes (z_1,z_2)) = U_1(x\otimes z_1) \oplus U_2(x\otimes z_2)$, 
is a normal trace preserving *-automorphism, and $\Phi = J_1^*\circ U \circ J$.
It follows that 
$\frak{D}_{{\rm qc},\cl D}(H)$ is convex.
The convexity of $\frak{D}_{{\rm loc},\cl D}(H)$ and $\frak{D}_{{\rm q},\cl D}(H)$
follows from the fact that, in the preceding argument, the commutativity 
(resp. finite dimensionality) is preserved under direct sums.

The claims for the sets $\frak{D}_{{\rm t},\cl D}^{\rm ex}(H)$ are similar;  consider the case where ${\rm t} = {\rm qc}$. 
With the notation from the previous paragraph, we let 
$D_i\in \cl D\bar\otimes\cl N_i$, be a unitary operator, 
such that $\Phi_i = \Phi_{D_i}$, $i = 1,2$. 
After making the canonical identification 
$\cl D\bar\otimes\cl N = \cl D\bar\otimes\cl N_1\oplus \cl D\bar\otimes\cl N_2$, we
have that $\Phi = \Phi_{D_1\oplus D_2}$.

We next turn to the claim for ${\rm t} = {\rm qa}$. 
Following the steps from the first paragraph it suffices to show that $ \cl N:=\cl N_{1}\oplus \cl N_{2}$ embeds in a trace preserving way into a matricial ultraproduct. We may assume that $ \cl N_{i} \hookrightarrow \cl M_{i}^{\frak{u}}$, $ i=1,2$ along the same free ultrafilter $ \frak{u}$ where $ \cl M_{1}^{\frak{u}}
= \oplus_{k\in \bb{N}}^{\infty}M_{n_{k}}/ \cl J_{\frak{u}}$ and  $ \cl M_{2}^{\frak{u}}= \oplus_{k\in \bb{N}}^{\infty}M_{m_{k}}/ \cl I_{\frak{u}}$
(see e.g. \cite[Section 3.6]{goldbring}). Identify  $ \cl N_{i} $ as a von Neumann subalgebra of $\cl M_{i}^{\frak{u}}$, $ i=1,2$ and denote by $ \tau_{i}$ the respective trace. Hence $ \cl M_{1}^{\frak{u}} \oplus \cl M_{2}^{\frak{u}}$ is a von Neumann algebra equipped with the trace $ \tau = \lambda_{1} \tau_{1} + \lambda_{2}\tau_{2}$. Then,
\begin{align*}
    (\oplus_{k\in \bb{N}}^{\infty}M_{n_{k}}/ \cl J_{\frak{u}}) \oplus ( \oplus_{k\in \bb{N}}^{\infty}M_{m_{k}}/ \cl I_{\frak{u}}) \cong  \oplus_{k\in \bb{N}}^{\infty}(M_{n_{k}} \oplus M_{m_{k}}) /\cl J'_{\frak{u}}
\end{align*}
where $\cl J'_{\frak{u}} = \{(x_{k}\oplus y_{k})_{k\in \bb N}: \lim_{\frak{u}}\tau_{k}(x^{*}_{k}x_{k}\oplus y^{*}_{k}y_{k})=0 \}$ with $ \tau_{k} = \lambda_{1}\tr_{n_{k}} + \lambda_{2}\tr_{m_{k}}$. Now the claim follows by embedding the latter ultraproduct into a matricial one \cite[Remark 11.32]{Pi2}.
\end{proof}

In the next result central for this section, 
Theorem \ref{p_closures} below, 
we will need a lemma about the behaviour of ultraproducts under 
tensoring. 
Let $(\cl D,\delta)$ be a finite
tracial von Neumann algebra, acting on the 
Hilbert space $H = L^2(\cl D,\delta)$. 
We equip $\cl D\bar\otimes\cl M_i$ with the tracial state $\delta\otimes\tau_i$, 
and note that the Hilbert space of the GNS construction arising from 
$\delta\otimes\tau_i$ coincides with the Hibertian tensor product 
$H\otimes K_i$, $i\in I$. 
Let $\tilde{\cl M} = \oplus^{\infty}_{i\in I} \cl D\bar\otimes\cl M_i$, 
and $\tilde{\cl J}_{\frak{u}}$ be the ideal of $\tilde{\cl M}$,
corresponding to the family 
$((\cl D\bar\otimes\cl M_i,\delta \otimes \tau_i))_{i\in I}$, 
defined analogously to (\ref{eq_Ju}). 
Let $\tilde{f}_{\frak{u}} : \tilde{\cl M}\to \bb{C}$ be the state, defined 
as in (\ref{eq_fu}), $\tilde{H}_{\frak{u}}$ be the Hilbert space 
arising from the GNS construction, applied to $\tilde{f}_{\frak{u}}$, and
$\tilde{\pi}_{f_{\frak{u}}} : \tilde{\cl M}\to \cl B(\tilde{H}_{\frak{u}})$
be the corresponding GNS representation.
Thus, the ultraproduct of the family 
$((\cl D\bar\otimes\cl M_i,\delta \otimes \tau_i))_{i\in I}$ along $\frak{u}$ is
*-isomorphic to the image $\tilde{\pi}_{\frak{u}}\left(\tilde{\cl M}/\tilde{\cl J}_{\frak{u}}\right)$ inside $\cl B(\tilde{H}_{\frak{u}})$.

The notation established in the last two paragraphs is used 
in the formulation and the proof of the next lemma. 
If $V$ and $W$ are vector spaces, we denote by 
$V\odot W$ their algebraic tensor product.

\begin{lemma}\label{l_tenuni}
The operator $\tilde{V} : \cl D\odot \cl M \to 
\tilde{\cl M}/\tilde{\cl J_{\frak{u}}}$, given by 
$\tilde{V}(a\otimes (x_i)_{i\in I}) = (a\otimes x_i)_{i\in I} + \tilde{\cl J}_{\frak{u}}$, 
annihilates $\cl D\odot \cl J_{\frak{u}}$ and thus induces an operator 
$V : \cl D\odot \left(\cl M/\cl J_{\frak{u}}\right) \to \tilde{\cl M}/\tilde{\cl J_{\frak{u}}}$, 
which is isometric with respect to the norms of 
$H\otimes K_{\frak{u}}$ and $\tilde{H}_{\frak{u}}$. 
Thus, $V$ extends to an isometry (denoted the same way)
\( 
V:H\otimes K_{\frak{u}}\longrightarrow\tilde H_{\frak{u}},
\) such that
\begin{equation}\label{eq_unitequi}
\cl D\bar\otimes \pi_{\frak{u}}\left(\cl M/\cl J_{\frak{u}}\right) \subseteq V^* \tilde{\pi}_{\frak{u}}\left(\tilde{\cl M}/\tilde{\cl J}_{\frak{u}}\right) V.
\end{equation}
In particular, 
$$\cl D\bar{\otimes} \left(\cl M/\cl J_{\frak{u}}\right) \hookrightarrow 
\tilde{\cl M}/\tilde{\cl J}_{\frak{u}}$$
via  a normal trace-preserving *-monomorphism. 
\end{lemma}

\begin{proof}
Suppose that $(x_i)_{i\in I}\in \cl J_{\frak{u}}$; then 
$\lim_{\frak{u}} \tau_i(x_i^*x_i) = 0$. 
For $a\in \cl D$, we have that 
$$(\delta\otimes\tau_i)((a\otimes x_i)^*(a\otimes x_i)) 
= \delta(a^*a)\tau_i(x_i^*x_i),$$
and hence
$\lim_{\frak{u}}(\delta\otimes\tau_i)((a\otimes x_i)^*a\otimes x_i)) = 0$,
implying that 
$(a\otimes x_i)_{i\in I}\in \tilde{\cl J}_{\frak{u}}$. 
Thus, $a\otimes (x_i)_{i\in I}\in \ker(\tilde{V})$, and hence 
$\cl D\odot \cl J_{\frak{u}}\subseteq \ker(\tilde{V})$. 

Write $q : \cl M\to  \cl M/\cl J_{\frak{u}}$ and $\tilde{q} : \tilde{\cl M}\to \tilde{\cl M}/\tilde{\cl J_{\frak{u}}}$
for the quotient maps. 
We show that the induced operator $V : \cl D\odot \left(\cl M/\cl J_{\frak{u}}\right) \to \tilde{\cl M}/\tilde{\cl J_{\frak{u}}}$ is inner product-preserving. 
Let $a_k, b_k \in \cl D$, $x_{i,k}, y_{i,k}\in \cl M_i$, and set 
$\tilde{x}_k = (x_{i,k})_{i\in I}$ and $\tilde{y}_k = (y_{i,k})_{i\in I}$, $k = 1,\dots,n$. 
Using the linearity of the limit along $\frak{u}$, we have that 
\begin{eqnarray*}
& & \left\langle \sum_{k=1}^n a_k\otimes q(\tilde{x}_k), 
\sum_{l=1}^n b_l\otimes q(\tilde{y}_l)\right\rangle 
= 
\sum_{k,l=1}^n \left\langle a_k, b_l\right\rangle
\left\langle q(\tilde{x}_k), q(\tilde{y}_l) \right\rangle\\
& = & 
\sum_{k,l=1}^n \delta(b_l^*a_k)
\lim\hspace{-0.05cm}\mbox{}_{\frak{u}}\tau_i(y_{i,l}^*x_{i,k})
= 
\lim\hspace{-0.05cm}\mbox{}_{\frak{u}}
\sum_{k,l=1}^n \delta(b_l^*a_k)\tau_i(y_{i,l}^*x_{i,k})\\
& = & 
\left\langle V\left(\sum_{k=1}^n a_k\otimes q(\tilde{x}_k)\right), 
V\left(\sum_{k=1}^n b_k\otimes q(\tilde{y}_k)\right)\right\rangle. 
\end{eqnarray*}
It follows that $V$ extends to an isometry (denoted in the same way) from 
$H\otimes K_{\frak{u}}$ into $\tilde{H}_{\frak{u}}$.

In the notation of the previous paragraph,

\begin{eqnarray*}
& &
\left\langle 
\tilde{\pi}_{\frak{u}}(\tilde{q}((c\otimes z_i)_{i\in I}))
V\left(\sum_{k=1}^n a_k\otimes q(\tilde{x}_k)\right), 
V\left(\sum_{k=1}^n b_k\otimes q(\tilde{y}_k)\right)\right\rangle\\
& = & 
\left\langle \hspace{-0.1cm}
\tilde{\pi}_{\frak{u}}(\tilde{q}((c \hspace{-0.05cm} \otimes \hspace{-0.05cm} z_i)_{i\in I}))
\hspace{-0.1cm}\left(\sum_{k=1}^n \tilde{q}((a_k\hspace{-0.05cm} \otimes \hspace{-0.05cm} x_{i,k})_{i\in I})\right), 
\left(\sum_{k=1}^n \tilde{q}((b_k \hspace{-0.05cm} \otimes \hspace{-0.05cm} y_{i,k})_{i\in I})\right)
\hspace{-0.1cm}\right\rangle\\
& = & 
\lim\hspace{-0.05cm}\mbox{}_{\frak{u}}
\sum_{k,l=1}^n \delta(b_l^*ca_k)\tau_i(y_{i,l}^*z_{i}x_{i,k})\\
& = & 
\left\langle 
c\otimes\pi_{\frak{u}}\left( q((z_i)_{i\in I})\right)
\left(\sum_{k=1}^n a_k\otimes q(\tilde{x}_k)\right), 
\sum_{k=1}^n b_k\otimes q(\tilde{y}_k) \right\rangle.
\end{eqnarray*}

It follows that
\( 
c\otimes\pi_{\mathfrak u}(q((z_i)_{i\in I}))
=
V^*\widetilde\pi_{\mathfrak u}
 \big(\widetilde q((c\otimes z_i)_{i\in I})\big)V.
\) 
By linearity and passage to the von Neumann algebra closures, we obtain
\( 
\mathcal D\bar\otimes
\pi_{\mathfrak u}(\mathcal M/\mathcal J_{\mathfrak u})
\subseteq
V^*\widetilde\pi_{\mathfrak u}
(\widetilde{\mathcal M}/\widetilde{\mathcal J}_{\mathfrak u})V.
\)
Equivalently, the map
\[
a\otimes q((x_i)_{i\in I})
\longmapsto
\widetilde q((a\otimes x_i)_{i\in I})
\]
extends to a normal trace-preserving $*$-monomorphism
\( 
\mathcal D\bar\otimes
(\mathcal M/\mathcal J_{\mathfrak u})
\hookrightarrow
\widetilde{\mathcal M}/\widetilde{\mathcal J}_{\mathfrak u}.
\)

The proof is complete. 
\end{proof}

Althought the following fact is certainly known, we give an elementary proof using the previous Lemma as it will be useful in the sequel. 
\begin{remark}\label{r_tensultra}
\rm
The tensor product of two ultraproducts embeds, in a normal
trace-preserving way, into an ultraproduct. Indeed, let
$(\cl M_i,\tau_{\cl M_i})_{i\in I}$ and
$(\cl N_j,\tau_{\cl N_j})_{j\in J}$ be families of tracial
von Neumann algebras, and let $\frak u$ and $\frak w$ be free
ultrafilters on $I$ and $J$, respectively. Let
$(\cl M^{\frak u},\tau_{\frak u})$ and
$(\cl N^{\frak w},\tau_{\frak w})$ denote the corresponding
ultraproducts. By Lemma \ref{l_tenuni}, there is a normal trace-preserving
$*$-monomorphism
\( 
\cl M^{\frak u}\bar\otimes\cl N^{\frak w}
\hookrightarrow
\widetilde{\cl N}/\widetilde{\cl J}_{\frak w},
\)
where
\( 
\widetilde{\cl N}
=
\bigoplus_{j\in J}^{\infty}
\cl M^{\frak u}\bar\otimes\cl N_j,
\)
and $\widetilde{\cl J}_{\frak w}$ is the induced ideal defined
as in the discussion preceding Lemma \ref{l_tenuni}.

Moreover, suppose that $I=J=\bb N$ and that $\cl M_k$ and
$\cl N_l$, $k,l\in\bb N$, are matrix algebras equipped with
their normalised traces. Applying Lemma \ref{l_tenuni} once more,
for every $l\in\bb N$ we obtain a normal trace-preserving
$*$-monomorphism
\( 
\cl M^{\frak u}\bar\otimes\cl N_l
\hookrightarrow
\left(
\bigoplus_{k\in\bb N}^{\infty}
\cl M_k\bar\otimes\cl N_l
\right)\big/
\widetilde{\cl J}_{\frak u,l}.
\)
The algebra on the right-hand side is an ultraproduct of matrix
algebras. Therefore, by \cite[Corollary 12.6]{Pi2},
the ultraproduct
\( 
\widetilde{\cl N}/\widetilde{\cl J}_{\frak w}
\)
embeds in a trace-preserving way into an ultraproduct of matrix
algebras. Consequently,
\( 
\cl M^{\frak u}\bar\otimes\cl N^{\frak w}
\)
also embeds in a trace-preserving way into an ultraproduct of
matrix algebras.
\end{remark}

\begin{theorem}\label{p_closures}
Let $H$ be a separable Hilbert space and $\cl D\subseteq \cl B(H)$ be a 
von Neumann algebra. 
The following hold:
\begin{itemize}
\item[(i)] 
the sets $\frak{D}_{{\rm loc},\cl D}^{\rm ex}(H)$ and
$\frak{D}_{{\rm qc},\cl D}^{\rm ex}(H)$ are closed in the point-weak* topology;

\item[(ii)] 
if $(\cl D,\delta)$ is a tracial von Neumann algebra acting on 
$H = L^2(\cl D,\delta)$ then 
the set $\frak{D}_{{\rm qa},\cl D}^{\rm ex}(H)$ is the closure 
of $\frak{D}_{{\rm q},\cl D}^{\rm ex}(H)$ in the point-weak* topology.
\end{itemize}
\end{theorem}

\begin{proof}
(i) Since $H$ is separable, the point-weak* topology on $\frak{D}_{{\rm qc},\cl D}^{\rm ex}(H)$ 
(which coincides with Arveson's BW topology as defined in \cite{Pa})
is metrisable; to show the closedness of $\frak{D}_{{\rm qc},\cl D}^{\rm ex}(H)$, we 
thus assume that $(\Phi_n)_{n\in \bb{N}}\subseteq \frak{D}_{{\rm qc},\cl D}^{\rm ex}(H)$ is 
a sequence, and $\Phi : \cl B(H)\to \cl B(H)$ is a 
unital completely positive map, such that 
$\Phi_n(x)\xrightarrow{n\to \infty} \Phi(x)$ in the weak* topology, for every $x\in \cl B(H)$. 
Let $(\cl N_n,\tau_n)$ be a finite von Neumann algebra, 
and $D_n\in \cl D\bar\otimes\cl N_n$ be a unitary such that 
$$\Phi_n(x) = (\id\otimes\tau_n)(D_n^*(x\otimes 1_{\cl N_n})D_n), \ \ \ x\in \cl B(H), \ 
n\in \bb{N}.$$
Fix a free ultrafilter 
$\frak{u}$ on $\bb{N}$ and let 
$(\cl N,\tau)$ be the ultraproduct of the family $\{(\cl N_n,\tau_n)\}_{n\in \bb{N}}$
along $\frak{u}$. 
Setting $\tilde{\cl N} = \oplus^{\infty}_{n\in \bb{N}}\cl N_n$ 
($\ell^{\infty}$-direct sum), we have that 
$\cl N$ is *-isomorphic to $\tilde{\cl N}/\cl J_{\frak{u}}$, 
where 
$$\cl J_{\frak{u}} = 
\left\{(z_n)_{n\in\bb{N}} : \lim\hspace{-0.05cm}\mbox{}_{\frak{u}} \tau_n(z_n^*z_n) = 0\right\}.$$ 

Let $\tilde{D} = \oplus^{\infty}_{n\in \bb{N}} D_n$; thus, 
$\tilde{D}\in \oplus^{\infty}_{n\in \bb{N}} \cl D\bar{\otimes}\cl N_n$ and, 
after a canonical identification, we view 
$\tilde{D}$ as an element of $\cl D\bar{\otimes}\tilde{\cl N}$.
Using the canonical identification 
$$\cl D\bar{\otimes}\tilde{\cl N} \equiv 
{\rm CB}(\cl D_*,\tilde{\cl N})$$
\cite[Corollary 7.1.5, Theorem 7.2.4]{er},
we associate with $\tilde{D}$ a completely contractive map 
$\tilde{\Gamma} : \cl D_* \to \tilde{\cl N}$. 
Let $\Gamma = q\circ \tilde{\Gamma}$, where $q$ is the quotient map; thus, 
$\Gamma : \cl D_* \to \cl N$ is completely contractive. 
Let $D\in\cl D\bar\otimes\cl N$ be the contraction that corresponds  to $\Gamma$; 
we thus have that $\Gamma(\omega)=L_\omega(D)$, $\omega \in \cl D_*$. 
We note that, by the definitions of the maps $\Gamma$ and $\tilde{\Gamma}$, 
we have 
$L_\omega (D)=q(L_\omega(\tilde D))$.


Let $D_{i,j} = L_{\epsilon_{j,i}}(D)$, and let 
$\tilde D_{i,j} = L_{\epsilon_{j,i}}(\tilde D)$. 
We have that 
$D_{i,j}\in\cl N$ and $\tilde D_{i,j}\in\tilde{\cl N}$. 
Write $\tilde D_{i,j} = (D_{n,i,j})_{n\in \bb{N}}$, and note that 
$D_{n,i,j} = L_{\epsilon_{j,i}}(D_n)$.

Observe that 
$$\langle\Phi_n(\epsilon_{k,l})e_j,e_i\rangle
=
\tau_n (L_{\epsilon_{j,i}}(D_n^*(\epsilon_{k,l}\otimes 1)D_n))
=
\tau_n(D_{n, k,i}^*D_{n,l,j}).$$
As $\tau(D_{k,i}^*D_{l,j})=\lim_{}\tau_n(D_{n,k,i}^*D_{n,l,j})$ and $\langle\Phi_n(\epsilon_{k,l})e_j,e_i\rangle\to
\langle\Phi(\epsilon_{k,l})e_j,e_i\rangle$, we obtain
$$\langle\Phi(\epsilon_{k,l})e_j,e_i\rangle=\tau(D_{k,i}^*D_{l,j}),$$
and hence $\Phi(x)=(\id\otimes\tau)(D^*(x\otimes 1)D)$.
As $\Phi(1)=1$, we get $(\id\otimes\tau)(D^*D)=1$ and hence $D^*D=1$, as $D$ is a contraction and $\tau$ is faithful. 
On the other hand, $\Phi$ is trace preserving and hence for any $x\in\cl S_1(H)$,
$$\tr(x) = \tr(\Phi(x))=(\tr\otimes\tau)((x\otimes 1_{\cl N})DD^*)$$
and hence $(\id\otimes\tau)(DD^*)=1$ and $DD^*=1$.

We thus showed that 
$\frak{D}_{{\rm qc},\cl D}^{\rm ex}(H)$ is closed in the point-weak* topology.
The fact that $\frak{D}_{{\rm loc},\cl D}^{\rm ex}(H)$ is closed in the point-weak* topology follows from the previous paragraph, together with the fact that 
the ultrapower of a family of abelian von Neumann algebras is 
an abelian von Neumann algebra. 

\smallskip

(ii) 
Suppose that $\Phi \in \frak{D}_{{\rm qa},\cl D}^{\rm ex}(H)$, and let
$\cl M^{\frak{u}}$ be a matricial ultraproduct arising from 
a family $(M_{p_n},\tr)_{n\in \bb{N}}$ of matrix algebras and a free ultrafilter $\frak{u}$, and 
$D\in \cl D\bar\otimes \cl M^{\frak{u}}$ be a unitary, such that 
$\Phi = \Phi_D$.  For $i,j\in\bb I$, set
\( 
D_{i,j}=L_{\epsilon_{j,i}}(D)\in\cl M^{\frak u}.
\)
Let $\widehat D$ denote the image of $D$ under the normal
trace-preserving $*$-monomorphism
\( 
\cl D\bar\otimes\cl M^{\frak u}
\hookrightarrow
\left(
\bigoplus_{n\in\bb N}^{\infty}
\cl D\bar\otimes M_{p_n}
\right)\big/\widetilde{\cl J}_{\frak u}
\)
provided by Lemma \ref{l_tenuni}. Since this monomorphism is
unital, $\widehat D$ is a unitary.

By \cite[Lemma 11.30]{Pi2}, $\widehat D$ has a unitary lift
\( 
\widetilde D=(D_n)_{n\in\bb N}
\in
\bigoplus_{n\in\bb N}^{\infty}
\cl D\bar\otimes M_{p_n},
\)
where each $D_n\in\cl D\bar\otimes M_{p_n}$ is unitary. Put
\( 
D_{n,i,j}=L_{\epsilon_{j,i}}(D_n).
\)
By the definition of the embedding in Lemma \ref{l_tenuni},
\( 
q((D_{n,i,j})_{n\in\bb N})=D_{i,j}.
\)
Consequently,
\[
\tau_{\frak u}(D_{k,i}^*D_{l,j})
=
\lim_{\frak u}
\tr_{p_n}(D_{n,k,i}^*D_{n,l,j}),
\qquad i,j,k,l\in\bb I.
\]
It follows that
\begin{equation}\label{conv}
\left\langle
\Phi(\epsilon_{k,l})e_j,e_i
\right\rangle
=
\lim_{\frak u}
\left\langle
\Phi_{D_n}(\epsilon_{k,l})e_j,e_i
\right\rangle,
\qquad i,j,k,l\in\bb I.
\end{equation}

As $(\Phi_{D_n})_{n\in \bb{N}}$ is bounded in $\cl B(\cl B(H))$ and bounded sets of $\cl B(H)$ are precompact in the weak* topology, there is a map $\Phi_{\frak u}\in \cl B(\cl B(H))$ which is a cluster point in the point weak* topology. By (\ref{conv}), $\Phi_{\frak u}=\Phi$, and hence $\Phi$ is in the closure of $\frak{D}_{{\rm q},\cl D}^{\rm ex}(H)$.

To complete the proof of (ii), it suffices to show that 
$\frak{D}_{{\rm qa},\cl D}^{\rm ex}(H)$ is closed in point-weak* topology. 
This follows from the arguments in (i), taking into account the fact that 
the class of matricial ultraproducts is closed under taking ultrapowers
(see the comments before \cite[Corollary 12.6]{Pi2}).
\end{proof}

We note that, in the case where $\cl D$ is a maximal abelian von Neumann algebra, 
say $\cl D \equiv L^{\infty}(X,\mu)$ acting by multiplication on $H = L^2(X,\mu)$, 
the set $\frak{D}_{{\rm qc},\cl D}(H)$ coincides with the absolutely dilatable 
measurable Schur multipliers over $X\times X$
(see Corollary \ref{Schur_mult}). 
We complement this with the next corollary regarding 
the different classes of 
absolutely dilatable measurable Schur multipliers.

\begin{corollary}\label{c_Schuhi}
Let $(X,\mu)$ be a standard measure space, and $\cl D \equiv L^{\infty}(X,\mu)$, acting by multiplication on the Hilbert space $H = L^2(X,\mu)$. Then 
\begin{itemize}
\item[(i)] 
$\frak{D}_{{\rm qc},\cl D}(H) = \frak{D}_{{\rm qc},\cl D}^{\rm ex}(H)$
and hence $\frak{D}_{{\rm qc},\cl D}(H)$ is closed in the point-weak* topology; 

\item[(ii)]
$\frak{D}_{{\rm qa},\cl D}(H) = \frak{D}_{{\rm qa},\cl D}^{\rm ex}(H)$
and hence $\frak{D}_{{\rm qa},\cl D}(H)$ coincides with the closure 
in the point-weak* topology of $\frak{D}_{{\rm q},\cl D}(H)$.
\end{itemize}
\end{corollary}

\begin{proof}
The arguments in the proof of Corollary \ref{Schur_mult} show that $\frak{D}_{{\rm t},\cl D}(H)=\frak{D}_{{\rm t},\cl D}^{\rm ex}(H)$ for $\rm t\in\{loc, q, qa, qc\}$. The statement now follows from Theorem \ref{p_closures}. 
\end{proof}

Write $\text{Aut}(\cl B(H))$ for the set of weak$^*$-continuous  automorphisms of $\cl B(H)$, that is, 
$\text{Aut}(\cl B(H))=\{x\mapsto u^*xu: u \text{ is unitary}\}$ (see for example \cite[II.5.5.14]{blackadar}). 
If $\cl D\subseteq \cl B(H)$ is a von Neumann algebra, denote by $\text{Aut}_{\cl D'}(\cl B(H))$ the automorphisms which are $\cl D'$-modular. 
Clearly, if $u^*xdu=u^*xud$ for all $x\in \cl B(H)$ and $d\in\cl D'$; taking $x=u$ we obtain $ud=du$, and hence $u\in\cl D$. Therefore, $$\text{Aut}_{\cl D'}(\cl B(H))=\{x\mapsto u^*xu: u\in\cl D \text{ is unitary}\}.$$ Write 
$\overline{{\rm conv}}({\rm Aut}_{\cl D'}(\cl B(H))$ for the closed convex hull of  
${\rm Aut}_{\cl D'}(\cl B(H))$ with respect to pointwise weak$^*$ topology.

\begin{lemma}\label{l_informu}
Let $H$ be a separable Hilbert space, 
$\cl D\subseteq \cl B(H)$ be a von Neumann algebra and 
$(X,\mu)$ be a standard measure space. 
Let $D\in \cl D\bar\otimes L^{\infty}(X,\mu)$, and 
write $d : X\to \cl D$ for the function, associated with $D$.
Then 
\begin{equation}\label{eq_informu}
\int_X \langle d(x)^*zd(x)\xi,\eta\rangle d\mu(x) =
\left\langle (z\otimes 1)D(\xi\otimes 1), D(\eta\otimes 1)\right\rangle \ \ \ z\in \cl B(H), \xi,\eta\in H.
\end{equation}
\end{lemma}

\begin{proof}
By polarisation and the fact that the positive operators
span $\cl B(H)$, in order to prove (\ref{eq_informu}), 
we may assume that $z\in \cl B(H)^+$ and $\eta = \xi$. 
Assume, without loss of generality, that $\|D\| = 1$.
Suppose first that $D = u\otimes h$ for some $u\in \cl D$ and some 
$h\in L^{\infty}(X,\mu)$. Then $d(x) = h(x)u$, $x\in X$, and hence 
\begin{eqnarray*} 
\left\langle (z\otimes 1)D(\xi\otimes 1), D(\xi\otimes 1)\right\rangle
& = & 
\left\langle zu\xi\otimes h, u\xi\otimes h)\right\rangle\\
& = &   
\langle u^*zu\xi,\xi\rangle \hspace{-0.2cm} 
\int_X \hspace{-0.2cm} |h(x)|^2 d\mu(x)
\hspace{-0.1cm} = \hspace{-0.2cm}
\int_X \hspace{-0.2cm} \langle d(x)^*zd(x)\xi,\xi\rangle d\mu(x).
\end{eqnarray*}  
By linearity, (\ref{eq_informu}) holds true if 
$D \in \cl D\odot L^{\infty}(X,\mu)$.

Now assume that $D$ is an arbitrary element of 
$\cl D\bar\otimes L^{\infty}(X,\mu)$. By the separability and the 
standardness assumptions, there exists a sequence 
$(D_n)_{n\in \bb{N}}$ in the unit ball of $\cl D\odot L^{\infty}(X,\mu)$, 
such that 
$D_n\to_{n\to\infty} D$ in the strong operator topology. 
Then 
\begin{equation}\label{eq_DntoD}
D_n(\xi\otimes 1)\to_{n\to \infty} D(\xi\otimes 1).
\end{equation}
Let $d_n : X\to \cl D$ be the function, canonically associated with 
$D_n$, $n\in \bb{N}$. 
By assumption, we have 
$$\int_X \|d_n(x)\xi - d(x)\xi\|^2 d\mu(x) \to_{n\to \infty} 0;$$
let $(n_k)_{k\in \bb{N}}\subseteq \bb{N}$ be such that 
$$\left\|d_{n_k}(x)\xi - d(x)\xi\right\| \to_{k\to \infty} 0 
\ \mbox{ for almost all } x\in X.$$
Using the Lebesgue Dominated Convergence Theorem, 
it now easily follows that 
$$\int_X \hspace{-0.1cm} \langle d_{n_k}(x)^*zd_{n_k}(x)\xi,\xi\rangle d\mu(x)
\to_{k\to \infty}
\int_X \hspace{-0.1cm} \langle d(x)^*zd(x)\xi,\xi\rangle d\mu(x),$$
and (\ref{eq_informu}) is established taking into account (\ref{eq_DntoD}). 
\end{proof}

\begin{theorem}\label{th_36}
    Let $H$ be a separable Hilbert space 
    and $\cl D\subseteq \cl B(H)$ be a von Neumann algebra. The following are equivalent for a weak$^*$ continuous, unital, completely positive map $\Phi:\cl B(H)\to\cl B(H)$:
    \begin{itemize}
    \item[(i)] 
    $\Phi\in \frak{D}_{{\rm loc},\cl D}(H)$;
    \item[(ii)] 
    $\Phi\in \frak{D}_{{\rm loc},\cl D}^{\rm ex}(H)$;
    \item[(iii)] $\Phi\in \overline{{\rm conv}}({\rm Aut}_{\cl D'}(\cl B(H))$.
    \end{itemize}
\end{theorem}
\begin{proof} 
(i)$\Rightarrow$(ii) 
By definition, there exists an abelian von Neumann algebra 
$\cl N\subseteq \cl B(K)$, equipped with a faithful (tracial) state $\tau_{\cl N}$ and a unitary $D\in \cl B(H)\bar\otimes \cl B(K) $ which returns to $\cl N$, such that $\Phi(z)=(\id\otimes\tau_{\cl N})(D^*(z\otimes 1_{\cl N})D)$. Let $U$ be an automorphism of $\cl B(H\otimes K)$ given by $U(x)=D^*xD$, $x\in \cl B(H\otimes K)$. We have $U(z\otimes 1_{\cl N})\in \cl B(H)\bar\otimes\cl N$ for any $z\in\cl B(H)$.
Without loss of generality, assume that 
$\cl N=L^\infty(X,\mu)$, acting by multiplication on the Hilbert space 
$K=L^2(X,\mu)$, where $(X,\mu)$ is a standard
probability measure space 
and $\tau_{\cl N}$ is given by 
integration against $\mu$. 
We will canonically identify $\cl B(H)\bar\otimes \cl N$
with the space $L_\sigma^{\infty}(X,\cl B(H),\mu)$ of all bounded weakly measurable 
$\cl B(H)$-valued functions.

Let $\{\epsilon_{i,j}\}_{i,j\in\mathbb N}$ be a  
matrix unit system in $\cl B(H)$, arising from an 
orthonormal basis, and fix $i_0\in\mathbb N$. 
Then $q := U(\epsilon_{i_0,i_0}\otimes 1)$ is a projection and hence there exists 
a measurable set $X_0\subseteq X$, such that 
$\mu(X\setminus X_0)=0$, and $q(x)^2=q(x) = q(x)^*$ 
whenever $x\in X_0$. Since $q(x)$ is a 
projection for $x\in X_0$, we have that 
$\tr(q(x))\in \mathbb N\cup\{0,\infty\}$. Moreover, $q(x)\ne 0$ almost everywhere, as otherwise there exists a set of non-zero measure $X_1$ such that $U(\epsilon_{i,j}\otimes 1)(x) = 0$, $x\in X_1$, for all $i,j\in\mathbb N$, contradicting that $U(1_{H\otimes K})=1_{H\otimes K}$. 
We can thus assume that $\tr(q(x))\geq 1$ for all $x\in X_0$. 
As $\int_X \tr(q(x)) d\mu(x) = \tau_{\cl N}(1) = 1$ and $\tau_{\cl N}$ is faithful, we obtain that $q(x)$ is a rank one projection for 
almost all $x\in X_0$. 
    Removing a set of measure zero we may 
    further assume that $U(\epsilon_{i,j}\otimes 1)(x)$ is a system of matrix units for every $x\in X_0$. 
    
    Let $\{e_k\}_{k=1}^\infty$ be an orthonormal basis in $H$
    that gives rise to the matrix unit system 
    $\{\epsilon_{i,j}\}$, and let $X_k=\{x\in X_0: q(x)e_k\ne 0\}$. Set $\eta(x)=q(x)e_k/\|q(x)e_k\|$, $x\in X_k\setminus (\cup_{i=1}^{k-1}X_i)$. We have that $X\setminus(\cup_{k=1}^\infty X_k)$ has measure zero, as otherwise $q(x)e_k=0$ for all $k\in\mathbb N$ and all $x\in X_0\setminus(\cup_{k=1}^\infty X_k)$, and as the latter set is non-empty 
    we reach a contradiction with the fact that $q(x)\ne 0$ on $X_0$.  
    Clearly, the function $x\mapsto\eta(x)$ is measurable and  $q(x)=\eta(x)\eta(x)^*$, for every $x\in \cup_{k=1}^\infty X_k$.  
    
For $x\in X_0$, let $d(x) : H\to H$ be the operator, given by 
$$d(x)\xi
=
\sum_{k=1}^\infty 
\left\langle U(\epsilon_{i_0,k}\otimes 1)(x)\xi,\eta(x)\right\rangle e_k, \ \ \ \xi\in H.$$
For $\xi,\zeta\in H$, we have
    \begin{eqnarray*}
\left\langle d(x)^*\epsilon_{i,j}d(x)\xi,\zeta \right\rangle
& = & \left\langle \epsilon_{i,j}d(x)\xi,\epsilon_{i,i}d(x)\zeta \right\rangle\\
& = & 
\left\langle U(\epsilon_{i_0,j}\otimes 1)(x)\xi,\eta(x) \right\rangle
\overline{\left\langle U(\epsilon_{i_0,i}\otimes 1)(x)\zeta,\eta(x)
\right\rangle}\\
& = & 
\left\langle U(\epsilon_{i_0,j}\otimes 1)(x)\xi, U(\epsilon_{i_0,i}\otimes 1)(x)\zeta\right\rangle
= 
\left\langle U(\epsilon_{i,j}\otimes 1)(x)\xi,\zeta \right\rangle\hspace{-0.12cm},
    \end{eqnarray*}
showing that $d(x)^*\epsilon_{i,j}d(x)
= U(\epsilon_{i,j}\otimes 1)(x)$, $x\in X_0$ and $d(x)^*d(x)=1$. 
On the other hand, direct verification shows that, if $\zeta\in H$ then 
$$d(x)^* \zeta = \sum_{k=1}^{\infty} 
\left\langle \zeta, e_k \right\rangle  U(\epsilon_{i_0,k}\otimes 1)(x)^* \eta(x). 
$$
Since $U$ is *-preserving, by deleting a null set if 
necessary, we may assume that 
$$U(\epsilon_{i_0,k}\otimes 1)(x)^* = U(\epsilon_{k,i_0}\otimes 1)(x), 
\ \ \ x\in \cup_{l=1}^{\infty} X_l.$$
Thus, whenever $x\in \cup_{l=1}^{\infty} X_l$ and $\zeta\in H$, we have
\begin{eqnarray*}
d(x)d(x)^* \zeta 
& = & 
\sum_{k=1}^{\infty} 
\left\langle \zeta, e_k \right\rangle  
d(x)(U(\epsilon_{i_0,k}\otimes 1)(x)^* \eta(x))\\
& = & 
\sum_{k=1}^{\infty} 
\left\langle \zeta, e_k \right\rangle  
\sum_{m=1}^{\infty} 
\left\langle
U(\epsilon_{i_0,m}\otimes 1)(x)U(\epsilon_{i_0,k}\otimes 1)(x)^* \eta(x)),
\eta(x)\right\rangle e_m\\
& = & 
\sum_{k=1}^{\infty} 
\left\langle \zeta, e_k \right\rangle  
\sum_{m=1}^{\infty} 
\left\langle
U(\epsilon_{i_0,m}\otimes 1)(x)U(\epsilon_{k,i_0}\otimes 1)(x) \eta(x)),
\eta(x)\right\rangle e_m\\
& = & 
\sum_{k=1}^{\infty} 
\left\langle \zeta, e_k \right\rangle  
\sum_{m=1}^{\infty} 
\left\langle
U(\epsilon_{i_0,i_0}\otimes 1)(x) \eta(x)),
\eta(x)\right\rangle e_k = \zeta.
\end{eqnarray*}
Hence $d(x)d(x)^*=1$, showing that $d(x)$ is a unitary.  
It follows that,
    for $z=\epsilon_{i,j}\in\cl B(H)$ and $\omega\in\cl S_1(H)$, we have
\begin{eqnarray*}
\langle\Phi(z), \omega\rangle
& = & 
\langle U(z\otimes 1),\omega\otimes 1\rangle_{\cl B(H)\otimes L^\infty(X),\cl S_1(H)\otimes L^1(X) }\\
& = & \int_X\langle U(z\otimes 1)(x),\omega\rangle_{\cl B(H),\cl S_1(H)}d\mu(x)\\
& = & 
\int_X\langle d(x)^*zd(x),\omega\rangle_{\cl B(H), \cl S_1(H)}d\mu(x),
    \end{eqnarray*}
    showing that 
\begin{equation}\label{eq_Phiz=}    
\Phi(z)=\int_Xd(x)^*zd(x)d\mu(x)
=
(\id\otimes \tau_{\cl N})(\tilde{D}^*(z\otimes 1)\tilde{D})
\end{equation}
for the unitary 
    $\tilde{D} \in \cl B(H)\bar\otimes L^\infty(X,\mu)$ 
    corresponding to the function $x\mapsto d(x)$. 
    As $\Phi$ is normal, we have the equality for all $z\in\cl B(H)$,
    that is, identity (\ref{eq_Phimod}) 
    in the proof of Theorem \ref{tar1} is satisfied 
    with $\tilde{D}$ in the place of $D$. 
    As in the proof of the implication (i)$\Rightarrow$(ii) 
    of Theorem \ref{tar1}, using the modularity of $\Phi$, 
    we can now show that 
    $\tilde{D}\in \cl D\bar\otimes \cl B(L^2(X,\mu))$; since 
    $\tilde{D}\in \cl B(H)\bar\otimes L^\infty(X,\mu)$, 
    we conclude that 
    $\tilde{D}\in \cl D\bar\otimes L^\infty(X,\mu)$.

(ii)$\Rightarrow$(i) is trivial.

(ii)$\Rightarrow$(iii) 
Let $\Phi\in \frak{D}_{{\rm loc},\cl D}^{\rm ex}(H)$ and 
let $D\in \cl D\bar\otimes L^{\infty}(X,\mu)$ be a unitary 
operator, such that $\Phi = \Phi_D$. 
We may assume that $X$ is endowed with a second countable 
compact Hausdorff topology and let $C(X)$ be the associated 
space of complex-valued continuous functions (\cite[Theorem 4.4.4]{murphy}). 
By Kaplansky's Density Theorem 
\cite[Corollary 5.3.7]{kadison-ringrose}
and the separability assumptions, there exists a sequence 
$(D_n)_{n\in \bb{N}}$ of unitary operators in the C*-algebraic tensor product 
$\cl D\otimes C(X)$, 
such that $D_n\to_{n\to\infty} D$ in the strong operator topology. 
By (\ref{eq_Phiz=}) and Lemma \ref{l_informu}, 
$\Phi_{D_n}\to_{n\to \infty} \Phi_D$ in the point-weak* topology.

By the previous paragraph, we may 
assume that $D\in \cl D\otimes C(X)$. 
In this case, the associated function $d : X\to \cl D$
takes values in $\cl D$ and is continuous. 
Thus, if $\xi,\eta\in H$ then the 
function $x\mapsto \langle d(x)^*zd(x)\xi,\eta\rangle$ is continuous. 
Using the fact that the convex combinations of Dirac measures are 
weak* dense in the set of all probability measures, 
we conclude that 
$\Phi$ is in the point-weak* closed hull of 
${\rm conv}({\rm Aut}_{\cl D'}(\cl B(H))$.


(iii)$\Rightarrow$(ii) 
It is clear that the maps $\Phi$ of the form $\Phi(z) = d^* z d$, where 
$d\in \cl D$ is a unitary, 
admit an exact factorisation via the trivial ancilla.
By Proposition \ref{p_convex}, 
${\rm conv}({\rm Aut}_{\cl D'}(\cl B(H))\subseteq \frak{D}_{{\rm loc},\cl D}^{\rm ex}(H)$. 
The claim now follows from Theorem \ref{p_closures} (i). 
\end{proof}

From the previous theorem, we have 
$\frak{D}_{\rm loc,\cl D} = \overline{\text{conv}}(\text{Aut}_{\cl D'}(\cl B(H))$. 
The elements of $\text{Aut}_{\cl D'}(\cl B(H))$ belong to 
$\frak{D}_{{\rm q},\cl D}(H)$ (with one-dimensional ancilla); thus,
by Proposition \ref{p_convex}, 
$$\frak{D}_{{\rm loc},\cl D}(H)\subseteq \overline{\frak{D}_{{\rm q},\cl D}(H)} \text{ and } \frak{D}_{{\rm q},\cl D}(H)\subseteq \frak{D}_{{\rm qc}, \cl D}(H).$$
If the Hilbert space $H$ is finite dimensional, 
then the convex hull of $\text{Aut}_{\cl D'}(\cl B(H)$ is closed and hence 
\begin{equation}\label{eq_findimi}
\frak{D}_{{\rm loc},\cl D}(H) \subseteq \frak{D}_{{\rm q},\cl D}(H).
\end{equation}
We do not know if (\ref{eq_findimi}) holds in the case where 
$H$ is infinite dimensional.

\begin{lemma}\label{tensor_dil}
    Let $H_i$ be a separable Hilbert space, 
    $\cl D_i\subseteq\cl B(H_i)$ be a von Neumann algebras, $\Phi_i\in\frak D_{{\rm t},\cl D_i}(H_i)$, $i=1,2$, and $t\in\{\rm loc,q,qa, qc\}$. Then $\Phi_1\otimes\Phi_2\in\frak D_{\rm t,\cl D_1\bar\otimes\cl D_2}(H_1\otimes H_2)$ . In particular, $\Phi_1\otimes\id_{\cl B(H_2)} \in\frak D_{\rm t,\cl D_1\bar\otimes \cl D_2}(H_1\otimes H_2)$. The same statements hold true for the classes 
    of maps that admit a ${\rm t}$-exact factorisation.
\end{lemma}

\begin{proof}
Assume first that $\rm t=qc$ and consider maps which admit exact $\rm qc$-factorisation. For $i=1,2$, let $(\cl N_i, \tau_i)$ be a finite von-Neumann algebra and $D_i\in\cl D_i\bar\otimes\cl N_i$ 
be a unitary such that
    $\Phi_i(x)=(\id\otimes\tau_i)(D_i^*(x\otimes 1_{\cl N_i})D_i)$. After an appropriate flip we think of  $D_1\otimes D_2$ as a unitary in $(\cl D_1\bar\otimes\cl D_2)\bar\otimes(\cl N_1\bar\otimes\cl N_2)$ and, for $x_i\in\cl B(H_i)$ have
\begin{eqnarray*}
& & 
\hspace{-1cm} (\Phi_1\otimes\Phi_2)(x_1\otimes x_2)\\
& = & 
((\id\otimes\id)\otimes(\tau_{\cl N_1}\otimes\tau_{\cl N_2})((D_1^*\otimes D_2^*)(x_1\otimes x_2\otimes 1_{\cl N_1}\otimes 1_{\cl N_2})(D_1\otimes D_2)),
\end{eqnarray*}
that is, $\Phi_1\otimes\Phi_2$ is  exact factorisable via the ancilla $\cl N_1\bar\otimes\cl N_2$. 
    Clearly, if   $\cl N_i$ are abelian (finite-dimensional) so is $\cl N_1\bar\otimes\cl N_2$. That $\cl N_1\bar\otimes\cl N_2$ is embedable into a matricial ultraproduct of so are $\cl N_1$ and $\cl N_2$ follows from Remark \ref{r_tensultra}. 
    
    The second statement follows from the fact that $\id_{\cl B(H_2)}$ is factorisable via the ancilla $\cl N=\mathbb C$. The proof for $\rm t$-factorisable maps is similar. 
\end{proof}

\begin{theorem}\label{th_clova}
\begin{itemize}
\item[(i)]
If $H$ is a separable Hilbert space and 
$\cl D\subseteq \cl B(H)$ be a maximal abelian von Neumann algebra
of dimension exceeding 10,  
then 
the inclusions 
$$\frak{D}_{{\rm loc},\cl D}(H)\cap \frak{D}_{{\rm q},\cl D}(H)\subseteq 
\frak{D}_{{\rm q},\cl D}(H)
\ \mbox{ and } \ 
\frak{D}_{{\rm q},\cl D}(H)
\subseteq \overline{\frak{D}_{{\rm q},\cl D}(H)}$$
are proper.

\item[(ii)] 
The equality 
$\frak{D}_{{\rm qa}, \cl D}(H) = \frak{D}_{{\rm qc}, \cl D}(H)$ holds 
for purely continuous maximal abelian von Neumann algebra $\cl D$ if and only if the Connes Embedding Problem
has an affirmative answer. 
\end{itemize}
\end{theorem}

\begin{proof} (i)
If $H$ is finite dimensional then the statement follows from \cite[Example 3.3]{hm1} and \cite[Theorem 4.1]{mr}.
We first assume that $\cl D$ is discrete and infinite
dimensional; thus, 
without loss of generality, $H = \ell^2$ and 
$\cl D = \ell^{\infty}$, 
acting by (pointwise) multiplication. 
By \cite[Example 3.3]{hm1}, there exists a 
matrix $B=(b(s,t))_{s,t=1}^6\in M_6(\mathbb C)$ such that the map $\Phi : M_6(\mathbb C)\to M_6(\mathbb C)$, $z\mapsto B\ast z$,
of Schur multiplication by the matrix $B$ 
is factorisable via a finite-dimensional ancilla but is not in the convex hull of $\text{Aut}(\cl B(\mathbb C^6))$. 
Therefore, there exist unitaries $D(s)\in M_n(\mathbb C)$,  $s=1,\ldots, 6$, such that $b(s,t)=\tr_n(D(s)^*D(t)) $. Extend $D$ to $\mathbb N$ by setting $D(s)=1$ for $s>6$ and let 
$\tilde b(s,t)=\tr_n(D(s)^*D(t))$, $s,t\in\mathbb N$.
Then the corresponding map 
$\Psi : \cl B(\ell^2)\to\cl B(\ell^2)$ is a factorisable map via a finite-dimensional ancilla. Assuming that $\Psi \in {\frak D}_{{\rm loc},\ell^\infty}(\ell^2)$, we obtain that $\tilde b(s,t)=\tau(U(s)^*U(t))$ for a unitary-valued map $U:\mathbb N\to \cl N$, where $\cl N$ is an abelian von-Neumann algebra and $\tau$ is its state. But this implies that $T_B$ is factorisable via an abelian ancilla $\cl N$, contradicting the choice of $B$. 
The case where $\cl D$ is discrete and finite dimensional is 
treated similarly. 

The proof of the second statement in (i) for the discrete case is similar and uses  \cite[Theorem 4.1]{mr}.

Assume that $\cl D$ is purely continuous, 
that is, $H = L^2(0,1)$ and 
$\cl D = L^{\infty}(0,1)$ (here $(0,1)$ is equipped with Lebesgue
measure). 
There is a unitary isomorphism $L^2(0,1)\simeq \mathbb C^6\otimes L^2(0,1)$ such that $\cl D$ is unitarily equivalent to $\ell^\infty([6])\otimes\cl D$. Indeed, 
let $p_i$ be the multiplication by $\chi_{X_i}$ where $X_i=((i-1)/6,i/6)$, $i=1,\ldots,6$. Then $U:\xi\in L^2(0,1)\mapsto (p_1\xi,\ldots,p_6\xi)\in \oplus_{i=1}^6 L^2(X_i)$ is unitary and $UdU^*(\xi_1,\ldots,\xi_6)=(d_1\xi_1,\ldots d_6\xi_6)$, where $d_i=d|_{X_i}$ for $d\in\cl D$. Furthermore, $L^2(X_i)$ is 
isomorphic to $L^2(0,1)$ via 
the conjugation by the unitary $U_i$, 
given by  $(U_i\xi)(x)=\xi((x+(i-1))/6)/6$, $\xi\in L^2(X_i)$ which also gives unitary equivalence of the multiplication masas on $L^2(X_i)$ and $L^2(0,1)$. 
We shall only show the first statement in (i), the other is proved in a similar way using \cite[Theorem 4.1]{mr} instead. Let $\Phi : M_6\to M_6$ be the map from the previous paragraph
and $\Psi : \cl B(\mathbb C^6\otimes H)\to 
\cl B(\mathbb C^6\otimes H)$ be the map, given by 
$\Psi(x) = (\Phi\otimes \id\mbox{}_{\cl B(H)})(x)$. 
By Lemma \ref{tensor_dil}, $\Psi\in \frak{D}_{{\rm q}, \ell^\infty([6])\otimes\cl D}(\mathbb C^6\otimes H)\equiv \frak{D}_{{\rm q}, \cl D}(H)$. 

Assume, by way of contradiction, 
that $\Psi\in \frak{D}_{{\rm loc}, \cl D}(H)$.
Let $\cl C \subseteq \cl B(K)$ 
be an abelian von Neumann algebra equipped with a  state 
$\tau_{\cl C}$, and $C\in \cl D\bar{\otimes} \cl C=L^\infty_\sigma((0,1),\cl C)$
be a unitary, such that $\Psi = \Phi_C$. 
Up to conjugation by a unitary operator, we have that that
$$\cl D\bar{\otimes} \cl C = \left(\ell^{\infty}([6])
\bar{\otimes}\cl D\right)\bar{\otimes} \cl C$$ and we can think of $C$ as element in $L^\infty_\sigma((0,1),\ell^{\infty}([6])\otimes\cl C)$.
Since $\Psi = \Phi_C$, for $z\in M_6(\mathbb C)$ we have  $$\Phi(z)\otimes I_{\cl B(H)}=\Psi(z\otimes I_{\cl B(H)})=(\id\hspace{-0.05cm}\mbox{}_{M_6(\mathbb C)}\otimes\id\hspace{-0.05cm}\mbox{}_{\cl B(H)}\otimes\tau_{\cl C})(C^*(z\otimes I_{\cl B(H)}\otimes 1_{\cl C})C) ,$$ so that after identifying $z\otimes I_{\cl B(H)}$, $z\in M_6(\mathbb C)$, with the constant function $s\mapsto z(s)=z$ in $L^\infty_\sigma((0,1),M_6(\bb{C}))$ we obtain  
$$\Phi(z)=\Phi(z)(s)=(\id\hspace{-0.05cm}\mbox{}_{M_6(\mathbb C)}\otimes \tau_{\cl C})(C^*(s)(z\otimes 1_{\cl C})C(s)) \ \mbox{ for almost all } s\in(0,1).$$
Moreover, $C(s)$ is unitary for almost all $s$. Therefore there exists $s_0$ such that $C(s_0)\in \ell^\infty([6])\otimes\cl C$ is unitary and 
$\Phi(z)=(\id\otimes \tau_{\cl C})(C^*(s_0)(z\otimes 1_{\cl C})C(s_0))$ for all $z\in M_6(\mathbb C)$, giving 
$\Phi \in \frak{D}_{{\rm loc}, \ell^{\infty}([6])}(\bb{C}^6)$, 
a contradiction.

Finally, assume that $\cl D$ is of mixed type, and write 
$H = H_1\oplus H_2$ and $\cl D = \cl D_1\oplus \cl D_2$, 
where $\cl D_1\subseteq \cl B(H_1)$ is 
purely continuous (and non-trivial), while 
$\cl D_2\subseteq \cl B(H_2)$ is totally atomic. 
Since $\Phi : \cl B(H_1\oplus H_2)\to \cl B(H_1\oplus H_2)$
is a $\cl D_1\oplus\cl D_2$-bimodule map, 
there exist linear maps $\Phi_{i,j} : \cl B(H_j)\to \cl B(H_i)$, 
$i,j = 1,2$, such that 
\begin{equation}\label{eq_Phideco}
\Phi\left( \left[
\begin{matrix} 
x_{1,1} & x_{1,2}\\
x_{2,1} & x_{2,2}
\end{matrix}
\right]\right) 
= 
\left[
\begin{matrix}  
\Phi_{1,1}(x_{1,1}) & \Phi_{1,2}(x_{1,2})\\
\Phi_{2,1}(x_{2,1}) & \Phi_{2,2}(x_{2,2})
\end{matrix}
\right], \ \ \ x_{i,j}\in \cl B(H_j,H_i), \ i,j = 1,2.
\end{equation}
Note that, since $\Phi$ is unital and completely positive, 
so is $\Phi_{1,1}$. 

\medskip

\noindent {\it Claim. } 
If ${\rm t}\in\{\rm loc,q,qa, qc\}$ and 
$\Phi\in\frak D_{\rm t,\cl D}(H)$ then 
$\Phi_{1,1}\in\frak D_{\rm t,\cl D_1}(H_1)$. 

\medskip

\noindent {\it Proof of Claim. }
We demonstrate the claim for ${\rm t} = {\rm qc}$; the 
rest of the cases are obtained verbatim. 
Assume that $(\cl N,\tau)$ is a tracial von Neumann algebra
and, by virtue of Theorem \ref{th_36}, that 
$D \in (\cl D_1\oplus\cl D_2)\bar\otimes\cl N$ is a unitary 
operator, such that 
\begin{equation}\label{eq_mixed}
\Phi\hspace{-0.1cm}\left(\hspace{-0.05cm}
\left[
\begin{matrix}  
x_{1,1} & x_{1,2}\\
x_{2,1} & x_{2,2}
\end{matrix}
\right]\hspace{-0.05cm}\right) \hspace{-0.1cm} = \hspace{-0.05cm} (\id\otimes\tau)
\hspace{-0.1cm}\left(\hspace{-0.1cm}D^*\hspace{-0.1cm}
\left(\hspace{-0.05cm}\left[
\begin{matrix}  
x_{1,1} & x_{1,2}\\
x_{2,1} & x_{2,2}
\end{matrix}
\right]
\hspace{-0.1cm}\otimes \hspace{-0.1cm} 1\hspace{-0.05cm}\right)\hspace{-0.1cm}D\hspace{-0.1cm}\right), 
\ \ 
\left[
\begin{matrix}  
x_{1,1} & x_{1,2}\\
x_{2,1} & x_{2,2}
\end{matrix}
\right]
\hspace{-0.1cm}\in \cl B(H_1\hspace{-0.05cm}\oplus \hspace{-0.05cm}H_2).
\end{equation}
Using the canonial identification 
$$(\cl D_1\oplus\cl D_2)\bar\otimes\cl N 
= 
(\cl D_1\bar\otimes\cl N) \oplus (\cl D_2\bar\otimes\cl N),$$
we write $D = D_1 \oplus D_2$, where 
$D_i\in \cl D_i\bar\otimes\cl N$, $i = 1,2$. 
Identity (\ref{eq_mixed}) now implies that 
$$\Phi_{1,1}(x_{1,1}) = (\id \otimes \tau_{\cl N})\left(D_1^{*}(x_{1,1} \otimes 1_{\cl N})D_1 \right), 
\ \ \ x_{1,1}\in \cl B(H_1),$$
and hence $\Phi_{1,1}\in \frak{D}_{{\rm qc},\cl D_1}(H_1)$. 

\medskip

Now assume, by way of contradiction, that 
$\frak{D}_{{\rm q},\cl D}(H) \subseteq \frak{D}_{{\rm loc},\cl D}(H)$. 
By the previous part of the proof, there exists 
$\Phi_{1}\in \frak{D}_{{\rm q},\cl D_1}(H_1)\setminus 
\frak{D}_{{\rm loc},\cl D_1}(H_1)$. 
Let $D := D_1\oplus I$, viewed as an element of 
$\cl D\bar\otimes\cl N$, and 
$\Phi := \Phi_D$; thus 
$\Phi : \cl B(H_1\oplus H_2)\to \cl B(H_1\oplus H_2)$.
We have that $\Phi\in \frak{D}_{{\rm q},\cl D}(H)$; 
in addition, $\Phi_{1,1} = \Phi_1$ (see the 
decomposition (\ref{eq_Phideco})). 
By assumption, 
$\Phi\in \frak{D}_{{\rm loc},\cl D}(H)$ and hence, 
by the Claim, $\Phi_{1}\in \frak{D}_{{\rm loc},\cl D_1}(H_1)$, 
a contradiction. 

Finally, suppose that 
$\frak{D}_{{\rm q},\cl D}(H)$ is closed in the point-weak* 
topology. 
Using the previous part of the proof, fix 
$\Phi_{1}\in \overline{\frak{D}_{{\rm q},\cl D_1}(H_1)}
\setminus \frak{D}_{{\rm q},\cl D_1}(H_1)$. 
Let $\left(\Phi_{1}^{(n)}\right)_{n\in \bb{N}}\subseteq \frak{D}_{{\rm q},\cl D_1}(H_1)$ be such that 
$\Phi_{1}^{(n)}\to_{n\to \infty} \Phi_1$ in the point-weak* 
topology. 
As in the previous paragraph, define $\Phi$ and 
$\Phi^{(n)}$ with the property that 
$\Phi_{1,1} = \Phi_1$ and $\Phi_{1,1}^{(n)} = \Phi_1^{(n)}$, $n\in \bb{N}$.
By separability and compactness, choose a 
subsequence $\left(\Phi^{(n_k)}\right)_{k\in \bb{N}}$, 
converging in the point-weak* topology to a unital 
completely positive map $\Psi$. 
It is clear that belongs to 
$\overline{\frak{D}_{{\rm q},\cl D}(H)}$ and that 
it is a $\cl D$-bimodule map, 
hence admitting a decomposition of the form (\ref{eq_Phideco}). 
Since $\Phi^{(n_k)}_1\to_{k\to \infty} \Psi_{1,1}$, 
we have that $\Psi_{1,1} = \Phi_1$. 
By assumption, $\Psi\in \frak{D}_{{\rm q},\cl D}(H)$; 
hence, by the Claim, $\Phi_1\in \frak{D}_{{\rm q},\cl D_1}(H_1)$,
a contradiction. 



(ii) It is clear that an affirmative answer to 
the Connes Embedding Problem implies the equality 
$\frak{D}_{{\rm qa}, \cl D}(H) = \frak{D}_{{\rm qc}, \cl D}(H)$ for any Hilbert space and any maximal abelian von Neumann algebra $\cl D$.
For the converse direction, let $H=L^2(0,1)$ and $\cl D=L^\infty(0,1)$.
Suppose that $\frak{D}_{{\rm qa}, \cl D}(H) = \frak{D}_{{\rm qc}, \cl D}(H)$ and that the Connes Embedding Problem has a negative answer. 
As in the proof of \cite[Theorem 3.7]{hm2}, 
let $k\in \bb{N}$ and $\Phi : M_k\to M_k$ be a 
factorisable Schur multiplier which does not belong to $\frak{D}_{{\rm qa}, \ell^\infty([k])}(\mathbb C^k)$. Identifying $L^2(0,1)$ and $\ell^2([k])\otimes L^2(0,1)$  and  defining  $\Psi =\Phi\otimes\id$ on $\cl B(\ell^2([k])\otimes L^2(0,1))$ as in the proof of (i), we obtain that $\Psi\in \frak{D}_{{\rm qc}, \cl D}(H)\equiv \frak{D}_{{\rm qc}, \ell^\infty([k])\otimes\cl D}(\mathbb C^k\otimes H)$ by Lemma \ref{tensor_dil}, but  $\Psi\notin\frak{D}_{{\rm qa}, \cl D}(H)$, as otherwise the arguments from (i) will give that 
$\Phi(z)=(\id\otimes\tau)(C^*(z\otimes 1_{\cl C})C)$ for a unitary $C\in\ell^\infty([k])\otimes\cl C$, where $\cl C$ is a von Neumann algebra that admits an embedding in a matricial ultraproduct, contradicting that $\Phi\not\in \frak{D}_{{\rm qa}, \ell^\infty([k])}(M_k(\mathbb C))$. 
\end{proof}


\end{document}